\documentclass[11pt]{article}

\usepackage[a4paper]{geometry}
\usepackage{lmodern}        
\usepackage{graphicx,color} 
\usepackage{enumitem}       
\usepackage{multirow}       

\usepackage{amsmath,amssymb,amsfonts,amsthm,mathtools,bm}
\usepackage{dsfont}
\usepackage{bbm}
\usepackage{comment}
\usepackage{a4wide}

\usepackage{tikz}
\usepackage{tikz-cd}
\usetikzlibrary{decorations.pathreplacing,decorations.pathmorphing}
\tikzset{snake it/.style={decorate, decoration=snake}}

\usepackage{subcaption}

\usepackage{hyperref}
\usepackage{cleveref}

\usepackage{listings}
\usepackage{xcolor}

\newtheorem{proposition}{Proposition}[section]
\newtheorem{theorem}[proposition]{Theorem}
\newtheorem{lemma}[proposition]{Lemma}

\newtheorem{definition}[proposition]{Definition}
\newtheorem{remark}[proposition]{Remark}
\newtheorem{example}[proposition]{Example}

\numberwithin{equation}{section}
\numberwithin{figure}{section}
\numberwithin{table}{section}

\crefname{proposition}{proposition}{propositions}
\Crefname{proposition}{Proposition}{Propositions}

\crefname{theorem}{theorem}{theorems}
\Crefname{theorem}{Theorem}{Theorems}

\crefname{lemma}{lemma}{lemmas}
\Crefname{lemma}{Lemma}{Lemmas}

\crefname{corollary}{corollary}{corollaries}
\Crefname{corollary}{Corollary}{Corollaries}

\crefname{definition}{definition}{definitions}
\Crefname{definition}{Definition}{Definitions}

\crefname{remark}{remark}{remarks}
\Crefname{remark}{Remark}{Remarks}

\crefname{example}{example}{examples}
\Crefname{example}{Example}{Examples}

\allowdisplaybreaks[4]

\renewcommand{\phi}{\varphi}
\renewcommand{\theta}{\vartheta}
\renewcommand{\epsilon}{\varepsilon}

\renewcommand{\d}[1]{\mathinner{\mathrm{d}{#1}}}

\newcommand{\dd}[1]{\mathinner{\mathrm{d}{#1}}}
\renewcommand{\L}[1]{\mathbf{L^#1}}
\newcommand{\Lloc}[1]{\mathbf{L^{#1}_{loc}}}
\newcommand{\W}[2]{\mathbf{W^{#1,#2}}}

\newcommand{\C}[1]{\mathbf{C^{#1}}}

\newcommand{\Cc}[1]{\mathbf{C_c^{#1}}}

\newcommand{\BV}{\mathbf{BV}}
\newcommand{\BVloc}{\mathbf{BV_{loc}}}

\newcommand{\modulo}[1]{{\left|#1\right|}}
\newcommand{\norma}[1]{{\left\|#1\right\|}}

\newcommand{\reali}{{\mathbb{R}}}
\newcommand{\naturali}{{\mathbb{N}}}

\newcommand{\tv}{\mathop\mathrm{TV}}
\newcommand{\essinf}{\mathop\mathrm{ess\,inf}}
\newcommand{\esssup}{\mathop\mathrm{ess\,sup}}

\DeclareMathOperator{\Lip}{Lip}

\begin{document}

\title{Control of Conservation Laws in the Nonlocal-to-Local Limit }
\author{Jan Friedrich\footnote{Chair of Optimal Control, Department for Mathematics, School of Computation, Information and Technology, Technical University of Munich, Boltzmannstraße 3, 85748 Garching b. Munich, Germany, \texttt{jan.friedrich@cit.tum.de}} \and Michael Herty\footnote{Chair of Numerical Analysis, Institute for Geometry and Applied Mathematics, RWTH Aachen University, Templergraben 55, 52056 Aachen, Germany, \texttt{herty@igpm.rwth-aachen.de}} \and Claudia Nocita \footnote{Corresponding author} \footnote{Department of Mathematics and Applications,
  University di Milano - Bicocca, via Roberto Cozzi 55, 20125 Milano, Italy, \texttt{c.nocita@campus.unimib.it}}}
\maketitle

\begin{abstract}
 We analyze a class of control problems where the initial datum acts as a control and the state is given by the entropy solution of (local) conservation laws by a \emph{nonlocal-to-local} limiting strategy. In particular we characterize the limit up to subsequence of minimizers to \emph{nonlocal} control problems as minimizer of the corresponding \emph{local} ones. 
Moreover, we also prove an analogous result at a discrete level by means of a Eulerian-Lagrangian scheme.

\medskip

  \medskip

  \noindent\textbf{MSC~2020:} 35L65, 49J20, 65M08.
  
\end{abstract}

\section{Introduction}
The main aim of the present work 
is to take advantage of the nonlocality in order to solve control problems applied to (local) conservation laws, where the initial datum of the PDE acts as a control. 
The starting point is the study of the following class of Cauchy problems for nonlocal conservation laws in one space variable $x \in \reali$
\begin{equation}\label{eq:23}
\begin{cases}
    \partial_t u_H + \partial_x \left(u_H v\left(u_H*\eta_H \right) \right) =0, \\
    u_H(0,\cdot) = u_o 
\end{cases}
\end{equation}
where the convolution kernel is $\eta_H(x)= \frac{1}{H}\eta\left(\frac{x}{H}\right)$, with $H >0$ and $\eta$ a fixed $\L1(\reali; \reali)$ nonnegative function with unit integral, which can be applied in the context of traffic flow modeling \cite{blandin2016,chiarello2018}.
In the limit $H \to 0^+$, the kernel functions $\eta_H$ converge weakly-* in the sense of measures to the Dirac delta and the problem~\eqref{eq:23} formally tends to the Cauchy problem for the (local) conservation law
\begin{equation}\label{eq:24}
\begin{cases}
    \partial_t u + \partial_x \left(u v(u) \right) =0, \\
    u(0,\cdot) = u_o .
\end{cases}
\end{equation}
The question of whether this limit rigorously holds was among the first raised in \cite{amorim2015} based on numerical experiments. However, these numerical findings have been explained by the built-in numerical viscosity of the underlying schemes \cite{colombo2021role}.
Nevertheless, since then, several papers rigorously investigate the hypothesis under which the limit \emph{nonlocal-to-local} is valid, that is, when the solution to~\eqref{eq:23} converges to the admissible entropy solution of~\eqref{eq:24} in the limit $H \to 0^+ $. Although the limit does not hold in general (in \cite{nonloc_loc_infty} the authors exhibit counterexamples in which the solution to the nonlocal problem~\eqref{eq:23} experiences a total variation blow-up in final time, ruling out the desired convergence), under certain hypotheses on the convolution kernel, the speed function and the initial datum, the limit is proved. In particular, under the assumptions of monotone decreasing $v$ and $\eta$ supported on the negative real semi-axis, which are common hypotheses in traffic flow models, positive results have been achieved (see e.g. \cite{zbMATH07213667, bressan2021, zbMATH08005762, nonloc_loc_exp, nonloc_loc_infty,Nonloc-to-loc-convex,friedrich2022singular, keimer2019approximation, keimer2023singulardisc, KEIMER2025129307}).
We also refer to the overviews \cite{colombo2023overview,keimer2023nonlocal} and the references therein.

The \textit{nonlocal-to-local} limit would be a strong tool to approach the study of local conservation laws as limit of nonlocal ones. 
In particular, we address here the optimization problem 
\begin{equation}\label{eq:42}
    \min_{u_0 \in \mathcal{U}_{ad}} \mathcal G(u_o) = \min_{u_0 \in \mathcal{U}_{ad}} 
    \left(
    \mathcal J (u(T)) + \mathcal{K}(u) + \mathcal{I}(u_o)
    \right),
\end{equation}
where the initial datum $u_0$ serves as a control variable within the admissible set $\mathcal{U}_{ad}$, and $ \mathcal G$ is the combination of three possible different objective functionals rather general depending on the entropy solution $u$ to~\eqref{eq:24}, on its profile at the final time $T>0$ and on the initial datum itself. 
The question was first raised in \cite{nonloc_loc_exp}: the limiting strategy may overcome the difficulties related to the non-differentiability of the semigroup of solutions to local conservation laws, which is strictly linked to the spontaneous formation of shocks in solutions to (local) conservation laws and the shift in their position. Note moreover that the general non-differentiability can be also shown in the class of $\C\infty$ initial data. The following example illustrates such property.
\begin{example}\label{example:1}
Fix $v(u)=1-u$ and let $u$ be the entropy solution to~\eqref{eq:24} associated to the initial datum $u_o \coloneqq \frac{1}{2}\chi_{[0,+\infty[}$. The entropy solution features a shock moving with speed $\frac{1}{2}$:
$$ u(t,x) = \begin{cases}
    0 & x<\frac{t}{2}, \\
    \frac{1}{2} & x>\frac{t}{2}.
\end{cases}$$
For fixed $\varepsilon >0$ we now perturb the initial datum $u_o$ in the direction $\delta u_o = \varepsilon \chi_{[0,+\infty[}$ and consider the initial datum $u_{o,\varepsilon} \coloneqq \left(\frac{1}{2} + \varepsilon\right)\chi_{[0,+\infty[} = u_o + \varepsilon \delta u_o$. The corresponding entropy solution $u_\varepsilon$ to~\eqref{eq:24} can be straightforwardly computed:
$$ u_\varepsilon(t,x) = \begin{cases}
    0 & x<\left(\frac{1}{2}-\varepsilon\right)t, \\
    \frac{1}{2}+\varepsilon & x>\left(\frac{1}{2}-\varepsilon\right)t.
\end{cases}$$
Hence, the difference quotient
$$\frac{u_\varepsilon - u } {\varepsilon} = \begin{cases}
    0 & x<\left(\frac{1}{2}-\varepsilon\right)t, \\
    \frac{1}{2\varepsilon} +1 & \left(\frac{1}{2}-\varepsilon\right)t <x < \frac{1}{2}t, \\
    1 & x> \frac{1}{2}t
\end{cases}$$
in the limit $\varepsilon \to 0^+$ converges weakly-* for all $t>0$ in the space of measures to $\mu_t =H(\cdot - \frac{t}{2}) + \frac{1}{2}\delta_{x=\frac{t}{2}}$, rather to a $\Lloc1$ function, showing the general non-directional differentiability in $\Lloc1(\reali; \reali)$ of the semigroup of entropy solutions to (local) conservation laws. 
\end{example}
The previous example motivates the study of the differentiability of the map $u_o \to u$ in the space of measures: in \cite{zbMATH01302241} the authors characterize the derivative as \emph{duality} solution to a PDE obtained linearizing~\eqref{eq:24}; in \cite{zbMATH00807784} the concept of \emph{generalized tangent vectors} is introduced to develop a variational calculus for piecewise Lipschitz continuous solutions to (local) conservation laws.
The technique is then used in \cite{zbMATH00889831, zbMATH05194890} to derive optimality conditions for distributed and boundary control problems for strictly hyperbolic conservation laws. However, as the calculus for \emph{generalized tangent vectors} is valid only under the \emph{a priori} assumption on the piecewise Lipschitz regularity of solutions to~\eqref{eq:24}, after constructing the optimal control it is necessary to verify that such control prevents the entropy solution to~\eqref{eq:24} to develop a gradient catastrophe. 
Finally, \cite{zbMATH01848475} deals with control of entropy solution to conservation laws with controlled initial data and source term by means of \emph{shift-variations} and \emph{shift-differentiability}. There, assuming piecewise $\C1$ initial data, the author proves the differentiability and an adjoint calculus for a large class of objective functionals without \emph{a priori} assumptions on the shock structure of solutions to~\eqref{eq:24}. 

On the other hand, the state of the art concerning control of nonlocal conservation laws is very much different. 
In particular, in \cite{zbMATH05908226} the authors are able to prove that under regularity assumptions on the initial data, the semigroup of solutions to~\eqref{eq:23} is strongly $\L1$ Gateaux differentiable in any sufficiently regular direction. Moreover, the derivative is proven to solve a linear nonlocal Cauchy problem, obtained by linearization of~\eqref{eq:23} and necessary conditions for optimality are derived for a large class of objective functionals. 
In \cite{zbMATH06371678} the authors study a control problem with controlled boundary and initial datum for a general $\L2$ tracking objective functional and analyze the corresponding adjoint system. 
In \cite{zbMATH05986509} the authors study the state controllability and nodal profile controllability for a conservation law with nonlocal velocity. 
In addition, the feedback stabilization of the latter model is studied in \cite{coron2013,chen2017global}.
Nevertheless, in most of the previous works the kernels are not anisotropic and the convolutions depend on the whole (bounded) domain.
For traffic flow models as in \eqref{eq:23} the boundary controllability has been investigated in \cite{bayen2021boundary}.

The comparison between control problems in the two settings -- related to local and nonlocal conservation laws -- motivates us to follow this strategy: taking advantage of the nonlocality to study in the \emph{nonlocal-to-local} limit the control problem related to~\eqref{eq:24}. 

The main aim of this paper is to establish, by a $\Gamma$-convergence argument proved in \Cref{teo:4}, the convergence in $\Lloc1(\reali; \reali)$ up to a subsequence of minimizers to \emph{nonlocal control problems} to minimizers to \emph{local} ones. This first main result is presented in \Cref{teo:6}.
Namely, for all $H > 0 $, we will investigate 
\begin{equation}\label{eq:43}
   \min_{u_o \in \mathcal{U}_{ad}} \mathcal{G}_H (u_o) \coloneqq
\min_{u_o \in \mathcal{U}_{ad}} \mathcal{J}(u_H(T)) + \mathcal{K}(u_H) +\mathcal{I}(u_o) 
\end{equation}
where $u_H$ is the solution to~\eqref{eq:23}, proving existence of minimizers in the appropriate admissible set $\mathcal{U}_{ad}$ and showing their convergence to minimizers to~\eqref{eq:42}. 
The analysis is developed in two analytical frameworks, which recall the ones prescribed by \cite{nonloc_loc_infty} and \cite{nonloc_loc_exp} and heavily rely on an improved result for \emph{nonlocal-to-local} limit in the case of the simultaneous convergence as $H \to 0^+$ of the kernel functions $\eta_H$ to the Dirac delta and the initial data to a fixed function in  $\Lloc1(\reali; \reali)$ in \Cref{teo:7}. 

The second aim of this work is to establish in \Cref{teo:8} the $\Gamma$-convergence in a discrete framework.
Instead of finite volume schemes, see e.g. \cite{blandin2016,chiarello2018,friedrich2023numerical,zbMATH07852679}, we adopt the Eulerian-Lagrangian scheme proposed in \cite{abreu_nonloc} with fixed spatial mesh $\Delta x$ to find an approximate solution to~\eqref{eq:23} and~\eqref{eq:24}. Indeed, in \cite{abreu_nonloc} the authors present numerical experiments which show that the scheme preserves the \emph{nonlocal-to-local} limit.
Similarly as in the continuum case, we prove for fixed $\Delta x$, the $\Gamma$-convergence of the discrete version of the functionals $\mathcal{G}_H$ and $\mathcal{G}$ (which will be denoted by the subscript $\Delta$) in the limit $H \to 0^+$ and show that, up to subsequence, minimizers to $\mathcal{G}_{\Delta, H}$ in an adequate discrete admissible set $\mathcal{U}_{\Delta, ad}$ converge in $\L1(\reali; \reali)$ to minimizers to $\mathcal{G}_{\Delta}$. This result is presented in \Cref{teo:9}. 
The $\Gamma$-convergence results here shown are summarized in the diagram \ref{scheme:diagonal}:
the solid lines are proven in \Cref{teo:4} and \Cref{teo:8}, the dotted lines are consequence of convergence of the approximated solution found by the numerical scheme to the solutions to ~\eqref{eq:23} and~\eqref{eq:24} and the dashed line is verified here numerically. We emphasize that a rigorous result concerning the validity of such limit, although beyond the aim of this paper, may be proven adopting a strategy similar to the one in \cite{zbMATH07852679}.
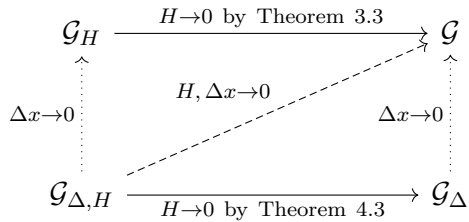
\begin{figure}
\centering
\begin{tikzcd}[column sep=5em, row sep=2em]
\mathcal{G}_H  
    \arrow[rr, "H \to 0~\mathrm{ by~\Cref{teo:4}}"]                           
  &  
  & \mathcal{G} 
\\ 
&  
&                                                            
\\
\mathcal{G}_{\Delta,H} 
    \arrow[uu, "\Delta x \to 0", dotted] 
    \arrow[rr, "H \to 0~\mathrm{ by~\Cref{teo:8}}"'] 
    \arrow[rruu, "{H,\, \Delta x \to 0}", dashed] 
  &  
  & \mathcal{G}_{\Delta} 
    \arrow[uu, "\Delta x \to 0", dotted]
\end{tikzcd}
\caption{ 
The $\Gamma$-convergence of functionals here proved: the solid lines represents the convergence in the \emph{nonlocal-to-local} limit in the continuum and discrete framework (\Cref{teo:4} and \Cref{teo:8}); the dotted lines are consequences of the Eulerian-Lagrangian scheme in the approximation of solution to~\eqref{eq:23} and~\eqref{eq:24}; at last, the diagonal dashed line is verified here numerically for some sequences $(H, \Delta x) \to 0$ (see \Cref{sec:diagonal convergence}).
}\label{scheme:diagonal}
\end{figure}

Some numerical simulations will be provided to illustrate the results claimed above. 

\section{General theory of nonlocal and local conservation laws}
In the present section we recall some important results concerning the theory of nonlocal and local conservation laws which will be useful in the sequel. 
We require some common assumptions for the speed function $v$ and the kernel function $\eta$:
\begin{enumerate}[label=$\bm{(v)}$]
\item \label{hyp:v} 
$v \in \C2(\reali)$ and $v''\leq 0$. Moreover, there exist $\delta, u_{max} > 0 $ such that $v(u_{max})= 0$ and $v'(r)\leq -\delta$ for all $r \in [0,u_{max}]$; 
\end{enumerate}

\begin{enumerate}[label=$\bm{(\eta)}$]
\item \label{hyp:eta}
$\eta(x) \geq 0$ for all $x \in \reali$, $\eta(x) = 0$ for every $x \in ]0,\infty]$ and $\int_\reali \eta(x) \dd{x} =1$. Furthermore $\eta$ is Lipschitz continuous on $]-\infty,0]$ and there exists a constant $D >0 $ such that $\eta(x) \leq D \eta'(x)$ for a.e. $x \in ]-\infty,0[$.
\end{enumerate}
The initial datum is chosen in the space
\begin{equation*}
    \mathcal{U} \coloneqq \{ u_o \in \L\infty(\reali;\reali) : u_o(x) \in [0, u_{max}] \, \text{for a.e.}\, x \in \reali \,\text{and} \,\tv(u_o ) < \infty \},
\end{equation*}
so that we are able to provide the definition of a \emph{solution} to~\eqref{eq:23}. 

\begin{definition}
Given $u_o \in \mathcal{U}$ and $H>0$, a function $u_H \in \C0([0,T]; \Lloc1(\reali; \reali))$ is a solution to the Cauchy problem~\eqref{eq:23} if, fixed $V(t,x) \coloneqq v( (u_H(t)*\eta_H)(x))$ with $\eta_H(x)=\frac{1}{H}\eta\left(\frac{x}{H}\right)$, $u_H$ is a weak solution to 
\begin{equation*}
    \begin{cases}
        \partial_t u_H(t,x) + \partial_x(u_H(t,x) V(t,x)) =0, \\
        u_H(0)=u_o.
    \end{cases}
\end{equation*}
\end{definition}
It can be shown that, under the previous definition, $u_H$ is also an \emph{entropy solution} (see, for example, \cite[Lemma 5.1]{zbMATH05908226}). 

For $H>0$ fixed, several papers deal with the well-posedness of the problem \eqref{eq:23}.
\begin{theorem}{\cite[Theorem 2.1]{zbMATH07213667}, \cite[Theorem 2.1]{nonloc_loc_exp}}
\label{teo:1}
Under assumptions \ref{hyp:v}-\ref{hyp:eta}, for any $u_o \in \mathcal{U}$ and $H>0$ there exists a unique solution to~\eqref{eq:23} $u_H \in \C0([0,T]; \Lloc1(\reali; \reali)) \cap \, \L\infty((0,T); \tv(\reali; \reali))$  and the following \emph{maximum principle} holds:
\[
\essinf_{x \in \reali} u_o (x) \leq u_H(t,x) \leq \esssup_{x \in \reali} u_o(x) \quad \text{ a.e. $t \in [0,T], \,x \in \reali$}.
\]
Moreover, the solutions to~\eqref{eq:23} are continuously dependent on the initial datum w.r.t. $\Lloc1$ norm, i.e. if $ \{u_o^n\}_n \subset \mathcal{U}$ such that $u_o^n \to u_o$ in $\Lloc1(\reali; \reali)$, then for all $t \in [0,T]$ $ u_H^n(t) \to u_H(t)$ in $\Lloc1(\reali; \reali)$, where $u_H^n$ is the solution to~\eqref{eq:23} with initial datum $u_o^n$.
\end{theorem}
Now, we report some classical results concerning the well-posedness of the Cauchy problem for the (local) conservation law~\eqref{eq:24}.
\begin{theorem}{\cite[Theorems 1,2,3,5]{MR267257}}\label{teo:5}
    Under assumption \ref{hyp:v}, for any $u_o \in \L\infty(\reali; \reali)$ there exists a (unique) Kru\v zkov solution to~\eqref{eq:24}. Furthermore, if $0 \leq u_o(x) \leq u_{max}$ for a.e. $x \in \reali$, then it holds that 
\[
0 \leq u(t,x) \leq u_{max} \quad \text{a.e. }\, x \in \reali, t \in [0,T]. 
\]
Additionally, if $u,\widehat u\in \L\infty ([0,T]\times\reali; \reali)$
  are {Kru\v zkov} solutions to~\eqref{eq:24} with initial data
  $u_o,\widehat{u}_o \in \L\infty (\reali; \reali)$, then, for a.e.
  $t \in [0, T]$,
  \begin{equation*}
    \begin{array}{c}
      \norma{u (t) - \widehat u (t)}_{\L1 ([a+\sigma t, b-\sigma \, t];\reali)}
      \leq
      \norma{u_o-\widehat{u}_o}_{\L1 ([a,b];\reali)}
    \end{array}
    \end{equation*}
where $\sigma= \norma{v}_{\L\infty([0,u_{max}]; \reali)} + u_{max}\norma{v'}_{\L\infty([0,u_{max}]; \reali)}$, $a<b$.
\end{theorem}
As mentioned in the introduction, several works address the question of the \emph{nonlocal-to-local} limit, i.e. whether the solutions to~\eqref{eq:23} converge to the entropy solution to~\eqref{eq:24} in the limit $H \to 0^+$. 
Here, we concentrate on two of those results.
In \cite{nonloc_loc_infty} the authors provide a positive result under general assumption on the kernel function $\eta$. 
We recall for completeness their main result adopting the following notation for $f:\reali \to \reali$
\begin{equation*}
    \Lip^-f \coloneqq -\inf_{x<y}\frac{f(y)-f(x)}{y-x}.
\end{equation*}
\begin{theorem}{\cite[Corollary 4]{nonloc_loc_infty}}
Assume \ref{hyp:v}-\ref{hyp:eta} and suppose $u_o \in \mathcal{U}$ satisfies $\Lip^-u_o < +\infty$ and $\inf_{x \in \reali} u_o>0$. Then for all $t \geq 0$ the family of solutions to the nonlocal Cauchy problem \eqref{eq:23} $u_H(t, \cdot)$ strongly converges in $\Lloc1(\reali)$ as $H \to 0^+$ to $u(t, \cdot)$, the entropy solution to \eqref{eq:24}.
\end{theorem}
The boundedness assumption of $\Lip^-$ of the initial datum prevents it from featuring negative jumps.
In particular, this technical restriction allows to apply Helly-Kolmogorov compactness Theorem in the proof to the \emph{nonlocal-to-local} convergence and to identify the limit as the entropy solution to~\eqref{eq:24} satisfying an Ole\u{\i}nik estimate.
On the other hand, the strict positivity condition $\inf_{x \in \reali} u_o(x) > 0$ is necessary to obtain total variation estimates on the solutions $u_H$ (see \cite[Theorem 6]{nonloc_loc_infty}), although it is not suitable for applications, for example in the context of traffic flow models.
A way to overcome such difficulty is proposed in \cite{nonloc_loc_exp, Nonloc-to-loc-convex}, where the compactness argument is applied to the convoluted terms $w_H \coloneqq u_H*\eta_H$, rather than to the solutions $u_H$. In the first reference, the authors deal with the case of the exponential kernel, i.e. $\eta(x)= e^x\chi_{(-\infty,0]}(x)$ and observe that the convoluted terms 
\begin{equation}\label{eq:31}
   w_H (t,x) \coloneqq \left( u_H(t)*\eta_H\right)(x) = \frac{1}{H} \int_x^{+\infty} u_H(t,y)e^{\left(x-y\right)/H} \dd y 
\end{equation}
obey to a transport equation with nonlocal source and feature a decrease of the spatial total variation in time. As a result, they are able to apply a compactness argument for the sequence $\left\{ w_H \right\}_H$ in $\C0([0,T];\Lloc1(\reali; \reali))$ and the identity $u_H = w_H - \frac{1}{H}\partial_x w_H$ (valid for the exponential kernel) allows to state a strong convergence result of $\{u_H\}_H$ to $u$, the entropy solution to~\eqref{eq:24}. 

\begin{theorem}{\cite[Theorem 2.4]{nonloc_loc_exp}}
Suppose that $u_o \in \mathcal{U}$,~\ref{hyp:v} holds and $\eta$ is the exponential kernel. Let for all $H>0$ $u_H \in \C0([0,T]; \Lloc1(\reali; \reali))$ be the solution to~\eqref{eq:23}. Then, the families $u_H, u_H*\eta_H \in \C0([0,T]; \Lloc1(\reali; \reali))$ converge in $\C0([0,T]; \Lloc1(\reali;\reali))$ as $H \to 0^+$ to the entropy solution to~\eqref{eq:24}.
\end{theorem}

For completeness, we add that in \cite{Nonloc-to-loc-convex} the authors observe that the total variation decreasing property of $w_H$ is strictly linked to the shape of the kernel function, which is required to be convex on $\reali_-$. Under such hypothesis and $u_o \in \mathcal{U}$, the authors prove that $w_H \to u$ in $\Lloc1([0,T]\times \reali; \reali)$ and $u_H \xrightharpoonup[]{*} u$ in $\L\infty([0,T]\times \reali; \reali)$ where $u$ is the entropy solution to~\eqref{eq:24}. 
In the following, we will concentrate on the settings in \cite{nonloc_loc_infty} and \cite{nonloc_loc_exp}.

\section{Control problem in the \emph{nonlocal-to-local} limit}
We devoted this section to the study of the control problem~\eqref{eq:42} by means of the \emph{nonlocal-to-local} limit. We here characterize as $H \to 0^+$ the limit (up to a subsequence) of minimizers to $\mathcal{G}_H$ (\eqref{eq:43}) as minimizer to $\mathcal{G}$. 
The tool we will exploiting is $\Gamma$-convergence of the functionals and an "improved" version of the \emph{nonlocal-to-local} limit. 

At first, we present a functional analysis result which will be exploited in the following. 
\begin{lemma} \label{lem:4}
    Let $f \in \L\infty(\reali; \reali)$ such that $\Lip^- f < + \infty$ and $f$ is pointwise well-defined for all $x \in \reali$. 
    Then $f \in \BVloc(\reali; \reali)$.
    In particular, for all $R > 0 $
    \begin{equation}\label{eq:33}
        \tv (f; [-R; R]) \leq 4 R\, \Lip^-f  + \sup_{x \in \reali} f(x) - \inf_{x \in \reali} f(x).
    \end{equation}
\end{lemma}
\begin{proof}
See \cite[Proof to Corollary 4]{nonloc_loc_infty}.
\end{proof}
We now present an "improved" result on the \emph{nonlocal-to-local} limit: as $n \to + \infty$ we study simultaneously the convergence of the kernel functions $\eta_{H_n}$ to the Dirac delta and the convergence of the initial data to a fixed $u_o$ in a suitable admissible set $\mathcal{U}_{ad}$.
Some remarks concerning the choice of $\mathcal{U}_{ad}$ are discussed below. 
\begin{theorem}[Improved result on the \emph{nonlocal-to-local} limit]\label{teo:7}
Let $0 < u_{min} \leq u_{max}, \, L,M>0$ and assume \ref{hyp:v}. 
Assume either 
\begin{enumerate} [label=$\textbf{(A)}$]
    \item \label{case A}
    \ref{hyp:eta} holds and  
\begin{equation}\label{eq:17}
 \mathcal{U}_{ad} \coloneqq \left\{ u_o \in \left(\L\infty \cap \BV \right)(\reali; \reali) : \norma{u_o}_{\L\infty(\reali; \reali)} \leq u_{max},\, \inf_{x \in \reali} u_o(x) \geq u_{min}, \, \Lip^-u_o \leq L\right\}   
\end{equation}
\end{enumerate}
\vspace{-0.5cm}
or 
\begin{enumerate} [label=$\textbf{(B)}$]
    \item \label{case B}
    $\eta$ is the exponential kernel and 
\begin{equation}\label{eq:32}
 \mathcal U_{ad} \coloneqq \left\{
 u_o \in 
 (\L\infty \cap \BV)(\reali; \reali) \colon
 \norma{u_o}_{\L\infty(\reali; \reali)} \leq u_{max}, \inf_{x \in \reali} u_o(x) \geq 0,\tv(u_o) \leq M
 \right\}
\end{equation}
\end{enumerate}
where, in both cases, we endow $\mathcal{U}_{ad}$ with the metric induced by the family of semi-norms in $\Lloc1(\reali;\reali)$.
For all $u_o \in \mathcal{U}_{ad}$ and any sequence $u_o^n \in \mathcal{U}_{ad}$ convergent to $u_o $ in $\Lloc1(\reali; \reali)$, call $u^n$ the solution to 
\begin{equation}\label{eq:35}
\begin{cases}
    \partial_t u^n + \partial_x \left(u^n v\left(u^n*\eta_{H_n} \right) \right) =0, \\
    u^n(0,\cdot) = u_o^n.
\end{cases}
\end{equation}
and let $u$ be the entropy solution to~\eqref{eq:24}.
Then,  $u^n \to u$ in $\C0([0,T]; \Lloc1( \reali; \reali)$.
\end{theorem}

We now add some remarks on the choice of the space $\mathcal{U}_{ad}$ in~\eqref{eq:17}. 
First, for $u_o \in \mathcal{U}_{ad}$ the condition $\tv(u_o) < + \infty $ is not redundant: indeed, the bound on the $\L\infty$-norm and on $\Lip^-(u_o)$ do guarantee a bound of $\tv(u_o)$ on all compact sets, but not on the entire domain, as the previous conditions does not prevent the function $u_o$ to eventually oscillate. 
Second, the condition $\inf_{x \in \reali} u_o(x) \geq u_{min}$ is necessary to prove the \emph{lower bound inequality} required for $\Gamma$-convergence. 
Indeed, if we substitute the foretold hypothesis with the mere assumption $\inf_{x \in \reali} u_o(x) > 0$, for each $u_o \in \mathcal{U}_{ad}$ we may find a convergent sequence $u_o^n$ in $\mathcal{U}_{ad}$ such that $\inf_{n \in \naturali} \inf_{x \in \reali} u_o^n(x) = 0 $, ruling out the uniform upper bound for $H_n$ such that~\eqref{eq:15} in Proof to \Cref{teo:4} holds. To fix the ideas, suppose $u_o \in \mathcal{U}_{ad}$ is right continuous and consider
\[
u_o^n \coloneqq 
\begin{cases}
    u_o(x) \quad & \text{if}\, x \geq -n, \\
    \min \left(1/n, \lim_{y \to (-n)^+}u_o(y) \right)  \quad  & \text{if}\, x < -n,
\end{cases}
\]
which meets $\Lip^-u_o^n \leq L, \tv(u_o^n) < + \infty, 0< u_o^n \leq u_{max}$ and $0 < \inf_{x \in \reali }u_o^n(x) \leq \frac{1}{n}$ for all $n \in \naturali$.
Third, the uniform bounds $\Lip^-u_o \leq L$ (Case \ref{case A}) and $\tv(u_o)\leq M$ (Case \ref{case B}) are necessary to perform the compactness argument which leads the singular limit to hold. 
\begin{proof}
\textbf{Case \ref{case A}.}
Set $u_o \in \mathcal{U}_{ad}$ and a sequence $u_o^n$ converging to $u_o$ in $\Lloc1(\reali; \reali)$.
For sake of comfort call $u^n \in \C0([0,T]; \Lloc1(\reali; \reali))$ the solution to the Cauchy problem~\eqref{eq:35}, which is well-posed owing to \Cref{teo:1}.
We now claim that there exists $u \in \C0([0,T]; \Lloc1(\reali; \reali))$ and a subsequence $u^{n_k}$ such that $u^{n_k}\to u$ strongly in $\C0([0,T]; \Lloc1(\reali; \reali ))$. This will follow from an application of Ascoli-Arzelà theorem and extends the result \cite[Corollary 4]{nonloc_loc_infty}.
Indeed, we first observe that the following holds: 
\begin{itemize}
    \item \textbf{For all $t \in [0,T]$ the sequence $\{u^n(t)\}_n$ is precompact in $\Lloc1(\reali; \reali)$.}
    \\
    By \cite[Theorem 3]{nonloc_loc_infty}, assuming $H_n \leq \frac{\inf_{x \in \reali}u_o^n}{2DL}$, the following uniform bound holds:
\begin{equation}\label{eq:15}
    \Lip^-u^n(t, \cdot) \leq \frac{L}{2 \delta L t +1}< \min\ \left(\frac{1}{2 \delta t}, L\right) \quad \text{for all}\, t\geq 0.
\end{equation}
A uniform upper bound for the choice of $H_n$ is then provided by $
    H_n \leq \frac{u_{min}}{2DL} $.
The inequality~\eqref{eq:15} implies a uniform bound on the total variation of $u^n(t, \cdot)$ on compact sets for all $t \in [0,T]$.
In fact, for all $R >0$ and $n \in \naturali$, by \Cref{lem:4}, one obtains that 
 \begin{equation}\label{eq:36}
         \tv (u^n(t,\cdot); [-R; R]) \leq 4R L  + u_{max} - u_{min}.
    \end{equation}
Recall moreover the maximum principle in \Cref{teo:1}. Then, Helly's Compactness Theorem (see \cite[Theorem 2.3]{Bressan}) allows to conclude that up to a subsequence $u^{n}(t, \cdot) \to u(t, \cdot)$ in $\Lloc1(\reali; \reali)$.
\item \textbf{$\{ u^n\}_n$ is uniformly equi-continuous in $\C0([0,T]; \Lloc1(\reali; \reali))$.}
\\
We observe here that the spatial bound on the total variation~\eqref{eq:36} can be transported to a local Lipschitz continuity in time of $u^n$ by exploiting the PDE it satisfies. 
We follow here the path of \cite[Theorem 4.3.1]{dafermos}. 
As $u^n$ is
  a distributional solution to~\eqref{eq:35}, for any
  $t_1,t_2 \in [0,T]$ with $t_1 < t_2$,
  $\phi \in \C1\left([0,T];\reali\right)$ with $\phi (t) = 1$ for all
  $t \in [0,t_2]$, $\psi \in \Cc1 ([-R,R];\reali)$ with
  $\norma{\psi}_{\L\infty (\reali;\reali)} \leq 1$, we have for
  $i=1,2$
  \begin{align*}
   0= & \int_{-R}^R u^n (t_i,x) \, \phi (t_i)\, \psi (x)\d{x} + \int_{t_i}^T \int_{-R}^R u^n (t,x) \, \partial_t \phi (t) \, \psi (x) \d{x} \d{t}
    \\
    &
    +
    \int_{t_i}^T \int_{-R}^R
    u^n (t,x) \, v\left(\left(u^n(t)*\eta_{H_n}\right)(x)\right) \, \phi (t) \, \partial_x \psi (x)
    \d{x} \d{t} 
  \end{align*}
Subtracting the two expressions, one gets, applying the definition of total variation \cite[Def. 3.4]{Ambrosio-Fusco-Pallara} and \cite[Ex. 3.17]{Ambrosio-Fusco-Pallara}
\begin{align}
\nonumber
&
    \int_{-R}^R \left(u^n(t_2, x) - u^n(t_1, x) \right) \psi(x) \dd x
    \\
    \nonumber
    =
    &
    \int_{t_1}^{t_2}\int_{-R}^R u^n(t,x) v\left(\left(u^n(t)*\eta_{H_n}\right)(x)\right) \, \phi (t) \, \partial_x \psi (x)
    \d{x} \d{t} 
    \\
    \nonumber
    \leq &
    \int_{t_1}^{t_2} \tv\left( u^n(t, \cdot)v\left(\left(u^n(t)*\eta_{H_n}\right)(\cdot)\right); [-R,R]
    \right) \dd t
    \\
    \label{eq:38}
    \leq &
    \int_{t_1}^{t_2} \norma{u^n(t)}_{\L\infty([-R,R]; \reali)} 
    \tv\left(
    v\left(\left(u^n(t)*\eta_{H_n}\right)(\cdot)\right); [-R,R]
    \right) \dd t
    \\
    \nonumber
    &
    + \int_{t_1}^{t_2} \tv\left(u^n(t, \cdot); [-R;R]\right) \norma{v}_{\L\infty(\reali; \reali)} \dd t .
\end{align}
We proceed now evaluating the term 
\begin{align}
\nonumber
    \tv\left(
    v\left(\left(u^n(t)*\eta_{H_n}\right)(\cdot)\right); [-R,R]
    \right) 
    \leq &
    \norma{v'}_{\L\infty([0,u_{max}]; \reali)} \tv\left(
   \left( u^n(t)*\eta_{H_n}\right)(\cdot); [-R,R]
    \right)
    \\
    \label{eq:37}
    \leq & 
    \norma{v'}_{\L\infty([0,u_{max}]; \reali)} \tv\left(
    u^n(t,\cdot); [-R,R]
    \right)
\end{align}
where we exploit \cite[Lemma 5.2]{Keimer-Pflug} and the fact that $$\tv\left(
   \left( u^n(t)*\eta_{H_n}\right)(\cdot); [-R,R]
    \right) \leq \norma{\eta_{H_n}}_{\L1(\reali; \reali)}\tv\left(
    u^n(t,\cdot); [-R,R]
    \right) .$$
Hence, inserting~\eqref{eq:36} and~\eqref{eq:37} in~\eqref{eq:38}, we obtain 
\begin{align*}
&
    \int_{-R}^R \left(u^n(t_2, x) - u^n(t_1, x) \right) \psi(x) \dd x 
    \\
    \leq
    &
    \int_{t_1}^{t_2}
    \left(
    u_{max} \norma{v'}_{\L\infty([0,u_{max}]; \reali)}
    +
    \norma{v}_{\L\infty([0,u_{max}]; \reali)}
    \right)\tv\left(
    u^n(t,\cdot); [-R,R]
    \right) \dd t
    \\
    \leq 
    &
    \left(
    u_{max} \norma{v'}_{\L\infty([0,u_{max}]; \reali)}
    +
    \norma{v}_{\L\infty([0,u_{max}]; \reali)}
    \right) \left( 4R L + u_{max} - u_{min}\right) \left( t_2 - t_1\right)
\end{align*}
and, passing to the $\sup$ of $\psi \in \Cc1([-R,R]; \reali)$ with $\norma{\psi}_{\L\infty(\reali; \reali)} \leq 1$ we conclude that $u^n \in \C0([0,T]; \Lloc1(\reali; \reali))$ is Lipschitz continuous in time, namely
\begin{equation}
    \norma{u^n(t_2)-u^n(t_1)}_{\L1([-R,R]; \reali)} \leq \mathcal L(R)  \left( t_2 - t_1\right)
\end{equation}
with $\mathcal L(R) \coloneqq \left(
    u_{max} \norma{v'}_{\L\infty([0,u_{max}]; \reali)}
    +
    \norma{v}_{\L\infty([0,u_{max}]; \reali)}
    \right) \left( 4RL + u_{max} - u_{min}\right) $.\\
As $\mathcal{L}(R)$ is independent of $n \in \naturali$, we have proved the uniform equi-continuity claimed.
\end{itemize}
Hence, an adaptation of Ascoli-Arzelà theorem guarantees that there exists a subsequence $u^{n_k}\to u$ in $\C0([0,T]; \Lloc1(\reali; \reali))$ with $u \in \C0([0,T]; \Lloc1(\reali; \reali))$ satisfying the Ole\u{\i}nik type estimate
\begin{equation}
    \Lip^- u(t, \cdot) < \frac{1}{2\delta t} \quad \text{for all }\, t\geq 0.
\end{equation}
due to the lower semi-continuity of the map $f \to \Lip^-(f)$ w.r.t. the $\Lloc1$ topology.

We conclude underlying that $u$ is actually a weak solution to~\eqref{eq:24}. Indeed, by the definition of weak solution to the Cauchy problem~\eqref{eq:35}, for all $k \in \naturali$, $u^{n_k}$ satisfies for all $\phi \in \Cc1(]-\infty, T[ \times \reali; \reali)$
      \begin{align*}
    \int_{0}^T \int_\reali u^{n_k} (t,x) \, \partial_t \phi (t,x) \d{x} \d{t}
    +
    \int_{0}^T \int_\reali
    u^{n_k} (t,x) \, v\left(\left(u^{n_k}(t)*\eta_{H_{n_k}}\right)(x)\right) \, \partial_x \phi (t,x)
    \d{x} \d{t} &
    \\
    +
    \int_\reali u^{n_k}_o (x) \, \phi (0,x)\,\d{x}
     =
    & 0.
  \end{align*}
Finally, the Dominated Convergence Theorem allows to pass the previous expression to the limit $k \to +\infty$, proving the claim. 
Since $u \in \C0([0,T]; \Lloc1(\reali; \reali))$ is a weak solution to~\eqref{eq:24} satisfying the Ole\u{\i}nik estimate, we can conclude that it is also the entropy one (see \cite[Chapter 8.5]{dafermos}).
Finally, as the previous computations holds for any sequence $\{u_o^n\}_n \subset \mathcal{U}_{ad}$, we can conclude by an absurd argument that the whole sequence $u^n$ converges to $u$ in $\C0([0,T]; \Lloc1(\reali; \reali))$. 
\\
\textbf{Case \ref{case B}.}
As before, consider $u_o \in \mathcal{U}_{ad}$ and let $u_o^n \in \mathcal{U}_{ad}$ be any convergent sequence. 
Given $u^n \in \C0([0,T]; \Lloc1(\reali; \reali))$ the solution to \eqref{eq:35}, recall the convoluted terms $w^n\coloneqq w_{H_n}$ as defined in~\eqref{eq:31}. 
By an adaptation, it is possible to show that the nonlocal terms $w^n$ feature the same properties proved in \cite{nonloc_loc_exp}. In particular, for all $n \in \naturali$, $w^n \in \W1\infty([0,T] \times \reali; \reali)$ and is a strong solution to the following transport equation with nonlocal source: 
\begin{equation*}
  \!\!  \left\{
 \begin{array}{rl}
 \!\!\!\! \partial_tw^n(t,x) \! + \! v\left( 
        w^n(t,x)
        \right) \partial_xw^n(t,x)
        & \!\!\!\!
        = 
        -\frac{1}{H_n} \int_x^{+\infty} \exp\left(\frac{x-y}{H_n}\right) 
        v'\left( w^n(t,y)\right)
        \partial_yw^n(t,y)
        w^n(t,y) \dd{y} \\
       w^n(0,x) 
        & \!\!\!\!
        =
        \frac{1}{H_n}\int_x^{+\infty} \exp\left(\frac{x-y}{H_n}\right)u_o^n(y) \dd{y}.
 \end{array}
\right.
\end{equation*}
Moreover, see \cite[Theorem 3.2]{nonloc_loc_exp}, the following uniform bound for the total variation in space of $w^n$ holds: 
\begin{equation*}
    \tv(w^n(t, \cdot)) \leq \tv(w^n(0,\cdot)) \leq \tv(u_o^n) \leq M \quad \text{for all} \, n \in \naturali.
\end{equation*}
Hence, we can adapt \cite[Theorem 4.1]{nonloc_loc_exp} and conclude that $\{ w^n\}_n$ is compactly embedded in $\C0([0,T]; \Lloc1(\reali; \reali))$, i.e. up to a subsequence, $w^n \to u$ in $\C0([0,T]; \Lloc1(\reali; \reali))$. The computation $\partial_x w^n(t,x)=\frac{1}{H_n}w^n(t,x) - \frac{1}{H_n}u^n(t,x)$, implies that for all $R >0$
\begin{equation*}
    \norma{w^n(t, \cdot) - u^n(t,\cdot)}_{\L1([-R, R]; \reali)} \leq H_n \norma{\partial_x w^n(t, \cdot)}_{\L1([-R,R]; \reali)} = H_n \tv(w^n(t, \cdot)) \leq H_n M.
\end{equation*}
Then, one proves that 
\begin{align*}
&
\lim_{n\to \infty}\norma{u^n - u}_{\C0([0,T]; \L1([-R,R]; \reali))} 
\\
\leq &
\lim_{n\to \infty}\norma{u^n - w^n}_{\C0([0,T]; \L1([-R,R]; \reali))}  + \lim_{n\to \infty}\norma{w^n - u}_{\C0([0,T]; \L1([-R,R]; \reali))}  
\\
\leq &
\lim_{n\to 0 } \, H_n M + \lim_{n\to \infty}\norma{w^n - u}_{\C0([0,T]; \L1([-R,R]; \reali))}  
\\
=
&
\,0.
\end{align*}
Similarly as done in Case \ref{case A}, one can prove that $u$ is the entropy solution to~\eqref{eq:24}.
\end{proof}
We are now ready to present the $\Gamma$-convergence result of the functionals $\{\mathcal{G}_{H}\}_H$ defined in~\eqref{eq:43} to $\mathcal{G}$ in~\eqref{eq:42} as $H\to 0^+$ which, by \Cref{teo:7},  we are able to prove for both settings \ref{case A} and \ref{case B}. 
 \begin{theorem}[$\Gamma$-convergence]\label{teo:4}
Let $0 < u_{min} \leq u_{max}, \, L,M>0$.
And suppose 
\begin{enumerate}[label=$\bm{(\mathcal{J})}$]
    \item
    \label{func_J}
    $\mathcal J: \Lloc1(\reali; \reali) \to \reali_{\geq 0}$ is continuous w.r.t. the $\Lloc1$ topology; 
\end{enumerate}
\begin{enumerate}
[label=$\bm{(\mathcal{K})}$]
    \item 
    \label{func_K}
    $\mathcal K: \Lloc1([0,T] \times  \reali; \reali) \to \reali_{\geq 0}$ is continuous w.r.t. the $\Lloc1$ topology; 
\end{enumerate}
\begin{enumerate}
[label=$\bm{(\mathcal{I})}$]
    \item \label{func_I}
    $\mathcal I: \Lloc1(\reali; \reali) \to \reali_{\geq 0}$ is lower semi-continuous w.r.t. the $\Lloc1$ topology. 
\end{enumerate}
Assume either \ref{case A} or \ref{case B} holds and consider the family of functionals $\mathcal{G}_{H}: \mathcal{U}_{ad} \to \reali $ indexed by $H > 0 $
\begin{equation}\label{eq:1}
 \mathcal G_H(u_{o}) \coloneqq \mathcal J(u_H(T)) + \mathcal K(u_H) + \mathcal I(u_o) \qquad \text{with $u_H$ solution to~\eqref{eq:23}}.
\end{equation}
Moreover, define $\mathcal G: \mathcal{U}_{ad} \to \reali$ as
\begin{equation}\label{eq:19}
   \mathcal{G}(u_o) \coloneqq \mathcal{J}(u(T)) + \mathcal{K}(u) + \mathcal{I}(u_o) \qquad \text{with $u$ entropy solution to~\eqref{eq:24}} .
\end{equation}
Then, the family of functionals $\left\{ \mathcal G_H \right\}_H $ $\Gamma$-converges to $\mathcal{G}$ in the limit $H \to 0$. 
\end{theorem}

The class of functionals $\mathcal{G}_H$ and $\mathcal G$ in \Cref{teo:4} is very general due to the choice of $\mathcal J, \mathcal K$ and $\mathcal I$, which account for \emph{observation at final time $t=T$} and \emph{distributed observation} functionals and (possible) regularization of the initial datum respectively. 
The class of functionals which fits into this description comprehends: 
\begin{itemize}
    \item \textbf{Tracking-type functional with observation at time $t=T$.} Given $u, u_d \!\in \! \L\infty(\reali; \reali)$
    \[
    \mathcal J (u) = \int_0^1 \modulo{u(x) - u_d(x)}^p \dd x, \qquad p \in[1, +\infty)
    \]
    \[
    \mathcal J(u) = \int_\reali \phi(u(x), u_d(x)) \dd x, \qquad \phi\in \Cc0(\reali^2; \reali_{\geq 0})
    \]
    \item \textbf{Smoothed tracking-type functionals with observation at time $t=T$.} Given $\gamma, \phi, \psi$ compactly supported non-negative smooth functions
    \[
    \mathcal J(u) = \int_\reali \gamma(x) \phi\left(\left( \psi * u\right)(x)\right) \dd x
    \]
    \item \textbf{Tracking-type functional with distributed observation.}
    \[
    \mathcal K(u) = \int_0^T \int_\reali \phi(u(t,x) - u_d(t,x)) \dd x \dd t \qquad \phi \in \Cc0(\reali^2; \reali_{\geq 0})
    \]
    \item \textbf{BV-regularization.} For $u_o \in \BV(\reali; \reali)$
    \[
    \mathcal I(u_o)= \tv(u_o).
    \]
\end{itemize}

\begin{proof}
As the space $\mathcal{U}_{ad}$ is endowed with a metric, it satisfies the first axiom of countability and for any sequence $H_n \xrightarrow[n\to\infty]{}0^+$, we are able to sequentially characterize the $\Gamma$-convergence of the functionals $\mathcal G_{H_n}$ to $\mathcal G$. Indeed, by \cite[Proposition 8.1]{zbMATH00735875} $ \mathcal G_{H_n}$ $\Gamma$-converges to $\mathcal G$ if and only if the following are satisfied: 
\begin{enumerate}
[label=$\bm{(\Gamma1)}$]
\item \label{Gamma conv 1} 
\textbf{Lower bound inequality.}
For every $u_o \in \mathcal{U}_{ad}$ and for every sequence $\{u_{o}^n\}_n \subset  \mathcal{U}_{ad}$ such that $u_{o}^n \to u_{o}$, it holds that
\begin{equation*}
    \mathcal G(u_{o}) \leq \liminf_{n \to \infty} \mathcal G_{H_n}(u^n_{o}).
\end{equation*}
\end{enumerate}
\begin{enumerate}
[label=$\bm{(\Gamma2)}$]
\item \label{Gamma conv 2}
\textbf{Upper bound equality.}
For every $u_{o} \in \mathcal{U}_{ad}$ there exists a sequence  $\{u_{o}^n\}_n$ converging to $u_{o}$ in $\mathcal{U}_{ad}$ such that
\begin{equation*}
    \mathcal G(u_{o}) = \lim_{n \to \infty} \mathcal G_{H_n}(u_{o}^n).
\end{equation*}
\end{enumerate}
We develop the proof distinguishing two steps. 
\\
\textbf{Proof to \ref{Gamma conv 1}.}
Set $u_o \in \mathcal{U}_{ad}$ and a sequence $\{u_o^n\}_n$ converging to $u_o$ in $\Lloc1(\reali; \reali)$.
In both cases \ref{case A} and \ref{case B}, one gets that $\{u^{n}\}_n \to u$ in $\C0([0,T];\Lloc1(\reali; \reali))$, where $u^{n}$ is the solution to \eqref{eq:35} and $u$ is the entropy solution to~\eqref{eq:24}.
The continuity properties of $\mathcal J$ and $\mathcal{K}$ and the lower semi-continuity of $\mathcal I$ yields that 
\begin{align}\label{eq:2}
     \mathcal G(u_o) = & \mathcal J(u(T)) + \mathcal K(u) + \mathcal I(u_o) 
     \\
     \nonumber
     \leq & \lim_{n \to + \infty}  \left(\mathcal J(u^{n}(T)) + \mathcal K(u^{n})\right)+ \liminf_{n \to +\infty} \mathcal I(u^{n}_o) = \liminf_{n \to + \infty} \mathcal G_{H_{n}}\left( u^{n}_o \right),
\end{align}
concluding the desidered inequality. 
\\
\textbf{Proof to \ref{Gamma conv 2}.}
Consider $u_o \in \mathcal{U}_{ad}$ and let $u_o^n \coloneqq u_o$ for all $n \in \naturali$. 
Hence, with the notation above, \Cref{teo:7} yields that $u^n \to u$ in $\C0([0,T]; \Lloc1(\reali; \reali))$. Then, the continuity of $\mathcal{J}$ and $\mathcal{K}$ allows to conclude the proof.
\end{proof}
$\Gamma$-convergence provides a powerful framework in the analysis of convergence of minimizers. At first we prove the existence of minimizers in the class $\mathcal{U}_{ad}$ for all functionals $\mathcal{G}_H$ and $\mathcal G$, which descends from the following $\Lloc1$ compactness property of the space $\mathcal{U}_{ad}$ and the lower semi-continuity of the functionals. 

\begin{lemma}[Compactness of $\mathcal{U}_{ad}$]\label{lem:2}
The spaces $\mathcal{U}_{ad}$ defined in~\eqref{eq:17} and~\eqref{eq:32} are compact w.r.t. the $\Lloc1$ topology. 
\end{lemma}

\begin{proof}
We will prove the statement in the case $\mathcal{U}_{ad}$ as in~\eqref{eq:17} as the remaining case is a direct consequence of Helly's Theorem.
Consider a sequence $\{u_{o}^n\}_n \subset \mathcal{U}_{ad}$ and observe preliminarily that $\tv(u^n_o)$ is uniformly bounded on all compact sets. In fact, by \Cref{lem:4}, one gets
 \begin{equation}
        \tv (u^n_o; [-R; R]) \leq 4 R\, L  + u_{max} - u_{min}.
    \end{equation}
Owing to Helly's Theorem (see \cite[Theorem 2.3]{Bressan}), there exists $u_o$ such that up to a subsequence
\begin{equation}\label{eq:11}
    \lim_{n \to \infty } u_o^n (x) = u_o(x) \quad \text{for all} \, x \in \reali,
\end{equation}
hence $u_o^{n}\to u_o$ in $\Lloc1(\reali; \reali)$.
We claim that $u_o \in \mathcal{U}_{ad}$: $u_o (x) \in [u_{min}, u_{max}]$ for all $x \in \reali $ due to \eqref{eq:11} and $\Lip^-u_o \leq L, \tv(u_o) < + \infty$ owing to the lower semincontinuity of the maps $f \to \Lip^- f$  and $f \to \tv(f)$ w.r.t. the $\Lloc1(\reali)$ topology. 
\end{proof}

\begin{proposition}[Existence of minimizers]\label{prop:1}
Let $0 < u_{min} \leq u_{max}, \, L, M >0$ and consider $\mathcal J, \mathcal I: \Lloc1(\reali; \reali) \to \reali_{\geq 0}$, $\mathcal K:\Lloc1([0,T]\times \reali; \reali)\to \reali_{\geq 0}$ lower semi-continuous w.r.t. the $\Lloc1$ topology. 
Assume either \ref{case A} or \ref{case B} holds. 
For each $H$ fixed, consider the functionals $\mathcal G_H, \mathcal G$ as in~\eqref{eq:1} and~\eqref{eq:19} defined on the space $\mathcal U_{ad}$. Then, there exist minimizers to the control problems
$$ \min_{u_o \in \mathcal{U}_{ad}} \mathcal G_H(u_o), \quad \min_{u_o \in \mathcal{U}_{ad}} \mathcal G(u_o).
$$
\end{proposition}

\begin{proof}
    We start by proving the existence of minimizers of $\mathcal{G}_H$ for fixed $H > 0$. 
    Due to the positiveness of the functionals involved, we can consider a minimizing sequence $\{u_o^n\}_n \subset \mathcal{U}_{ad}$, which, by the compactness property in \Cref{lem:2}, converges up to a subsequence to $u_o \in \mathcal{U}_{ad}$. 
    Call for all $n \in \naturali$ $u^n, u \in \C0([0,T]; \Lloc1(\reali; \reali))$ the solutions to~\eqref{eq:23} with initial datum $u^n_o$ and $u_o$ respectively. 
    Hence, \Cref{teo:1} yields that  $u^n(t) \to u(t)$ in $\Lloc1(\reali;\reali)$ for all $t \in [0,T]$ and, using the lower semi-continuity of $\mathcal J, \mathcal K,\mathcal I$, one gets
\begin{equation}
    \mathcal{G}_{H}(u_o) \leq \liminf_{n \to \infty} \mathcal G_H (u_o^n) = \inf_{f \in \mathcal{U}_{ad}}\mathcal G_H, 
\end{equation}
proving that $u_o \in \mathcal{U}_{ad}$ is a minimizer to $\mathcal{G}_H$. 
Similarly for the local counterpart, the functional $\mathcal{G}$ admits a minimizer in $\mathcal U_{ad}$ due to the $\Lloc1$-continuity on the initial datum of entropy solutions to~\eqref{eq:24} reported in \Cref{teo:5}.
\end{proof}

Once the existence of minimizers is proven, the $\Gamma$-convergence of functionals asserts that accumulation points of the sequence of minimizers are actually minimizers for the $\Gamma$-limit functional. More precisely, we recall the following theorem from \cite{zbMATH00735875}. 

\begin{theorem}[{\cite[Corollary 7.20]{zbMATH00735875}}]\label{teo:2}
    Consider a sequence $\{H_n\}_n \to 0$ and assume that $ \left\{\mathcal G_{H_n}\right\}_n$ $\Gamma$-converges to $\mathcal G$ in a topological space $X$. Let for every $n \in \naturali$ $u_o^n$ be a minimizer of $\mathcal G_{H_n}$ in $X$. If $u_o$ is a cluster point of $\{u_o^n\}_n$, i.e. if there exists a subsequence $\{u^{n_k}_o\}_k $ convergent to $u_o$ in $X$, then $u_o$ is a minimizer of $\mathcal G$ in $X$ and 
    \[
    \mathcal G(u_o) = \limsup_{n \to \infty}  \mathcal G_{H_n} (u_o^n).
    \]
    If $\{u_o^n\}_n$ converges to $u_o \in X$, then $u_o$ is a minimizer of $\mathcal G$ in $X$ and
    \[
        \mathcal G(u_o) = \lim_{n \to \infty} \mathcal G_{H_n}(u_o^n).
    \]
\end{theorem}
Then, an application of the previous theorem straightforwardly implies the main result of this section. 
\begin{theorem}[Convergence of minimizers]\label{teo:6}
 Let $0 <u_{min} \leq u_{max}$, $L, M>0 $, $H_n \to 0^+$ and assume either \ref{case A} or \ref{case B} holds. Consider the functionals $\left\{ \mathcal G_{H_n}\right\}_{H_n}$, $\mathcal G$ as in \Cref{teo:4} and let $u_{o}^n$ be a minimizer of $\mathcal{G}_{H_n}$ in $\mathcal{U}_{ad}$ for all $n \in \naturali$. Then, up to a subsequence, $u_o^n \to u_o$ where $u_o$ is a minimizer of $\mathcal G$ in $\mathcal{U}_{ad}$. 
 Moreover, \[
    \mathcal G(u_o) = \limsup_{n \to \infty}  \mathcal G_{H_n} (u_o^n).
    \]
\end{theorem}

\begin{proof}
    The proof immediately follows from the compactness property of $\mathcal{U}_{ad}$ in $\Lloc1(\reali; \reali)$ and \Cref{teo:2}. 
\end{proof}

\section{Discrete Control Problem in the \emph{nonlocal-to-local} limit}
\subsection{The Eulerian-Lagrangian scheme}
The present section is devoted to the description of the Eulerian-Lagrangian scheme to numerically solve~\eqref{eq:23} and~\eqref{eq:24}. 
The scheme, which we recall for clarity of exposition, was first introduced for (local) balance laws in \cite{abreu_loc} and extended to the nonlocal framework in \cite{abreu_nonloc}.
We summarize here the construction of the scheme to solve the Cauchy problem for the local conservation law~\eqref{eq:24}.

The core of the scheme is to trace the evolution in time of the average of the solution $u$ on spatial cells whose boundaries are delimited by \textit{no-flux} curves, i.e. curves along which the flux vanishes. 
Consider a uniform spatial grid with mesh size $\Delta x > 0 $ defined by the cell midpoints $x_j \coloneqq j \Delta x$, $j \in \mathbb{Z}$. Similarly, consider the temporal grid $t^m \coloneqq m \Delta t$ for $m\in \naturali$ with temporal step $\Delta t> 0 $ to be decided according to a suitable CFL condition. 

For all $j,m$ call $u_j^m$ the spatial average at time $t^m$ of the solution $u$ over the cell $C_j \coloneqq \left [ x_{j-\frac{1}{2}}, x_{j+\frac{1}{2}}\right]$ where $x_{j+\frac{1}{2}}\coloneqq (j+\frac{1}{2})\Delta x$. Namely, 
$$
u_j^m \coloneqq \frac{1}{\Delta x} \int_{x_{j-\frac{1}{2}}}^{x_{j+\frac{1}{2}}} u(t^m, x) \dd{x}.
$$
We call $\sigma_j^m$ the \textit{no-flux} curve arising from the grid point $x_j$ at time $t^m$ and characterize it as the solution to an ODE to be determined below.
Introduce the region $D_j^m$ delimited by the \textit{no-flux} curves $\sigma_j^m$ and $\sigma_{j+1}^m$, i.e.
\begin{equation}\label{eq:49}
  D_j^m \coloneqq \left\{ (t,x) \in \reali_+ \times \reali \, : \, \sigma_j^m(t) \leq x \leq  \sigma_{j+1}^m(t), \, t \in [t^m, t^{m+1}]
\right\}.  
\end{equation}
The nomenclature \textit{no-flux} curve is evidently motivated by the following computations.
First, integrate~\eqref{eq:24} on the region $D_j^m$
\begin{align*}
0 =
    \int_{D_j^m} \left( \partial_t u(t,x) + \partial_x \left( u(t,x) \, v(u(t,x))\right)\right) \dd{x} \dd t
    =
    \int_{D_j^m} \nabla_{x,t} \cdot 
    \left[
    \begin{array}{cc}
        u(t,x) \, v(u(t,x))   \\
         u(t,x)
    \end{array} \right] \dd{x} \dd t
\end{align*}
and apply the Divergence Theorem to conclude 
\begin{align}\label{eq:14}
0 = &\int_{\partial D_j^m} \left[
    \begin{array}{cc}
        u \, v(u)   \\
         u
    \end{array} \right]\cdot \bm{n} \dd{\partial D_j^m}
\end{align}
where $\partial D_j^m$ denotes the boundary of the integration region and $\bm{n}$ the external normal vector evaluated at the points of the boundary.
The \textit{no-flux} curves are defined in such a way that the line integral over $\sigma_j^m$ and $\sigma_{j+1}^m$ is actually zero.
To this aim, let $\tau$ be the parameter describing the curve $\tau \to ( \sigma_j^m(\tau), t(\tau))$ and note that the integral~\eqref{eq:14} over such line vanishes if the tangent vector to the curve is parallel to the vector field 
$\left[\begin{array}{cc}
        u \, v(u)   \\
         u
\end{array}
\right]$
along  the curve, i.e. if there exists $\lambda \in \reali,\, \lambda \neq 0 $ such that 
\begin{equation}\label{eq:12}
\begin{cases}
  \frac{\dd \sigma_j^m}{\dd \tau} (\tau)& = \lambda \,  u\left(t(\tau), \sigma_j^m(\tau )\right) \, v\left( u\left(t(\tau), \sigma_j^m(\tau)\right)\right), \\
  \frac{\dd t} {\dd \tau}(\tau) & = \lambda\,  u\left(t(\tau), \sigma_j^m(\tau )\right).
\end{cases}
\end{equation}
Eliminating $\tau$ in the family of solutions $(\sigma_j^m(\tau), t(\tau))$ to~\eqref{eq:12} leads to the following ODE describing the trajectory of $\sigma_j^m$: 
\begin{equation}\label{eq:13}
\begin{cases}
    \frac{\dd \sigma_j^m}{\dd t}(t) &=  v 
    \left(
    u\left(t, \sigma_j^m(t)\right)\right) \qquad t \in[t^m, t^{m+1}],  \\
    \sigma_j^m(t^m) &= x_j.
\end{cases}
\end{equation}
Hence, motivated by the no-flux property along the \textit{no-flux} curves described above, we prescribe that each $\sigma_j^m$ satisfies~\eqref{eq:13}.
Then,~\eqref{eq:14} reduces to
\begin{align}\label{eq:16}
0 = &\int_{\sigma_{j}^m(t^{m+1})}^{\sigma_{j+1}^m(t^{m+1})} u(t^{m+1}, x) \dd{x}   - \int_{x_{j}}^{x_{j+1}} u(t^m, x) \dd{x}.
\end{align}
Making use of the moving regions $D^m_j$ "impermeable" to the flux, we are able to follow the evolution of the spatial average of $u$ on such domains. In particular, we introduce the spatial average of the solution $u$ at time $t^{m+1}$ on the interval $[\sigma_{j}^m(t^{m+1}), \sigma_{j+1}^m(t^{m+1})]$:

\begin{flalign*}
&&
u_{j+\tfrac{1}{2}}^{m+1}
\coloneqq &
\frac{1}{\sigma_{j+1}^m(t^{m+1}) - \sigma_{j}^m(t^{m+1})}
\int_{\sigma_{j}^m(t^{m+1})}^{\sigma_{j+1}^m(t^{m+1})} 
u\bigl(t^{m+1}, x\bigr)\,\mathrm{d}x \\
&&
\quad= &
\frac{1}{\sigma_{j+1}^m(t^{m+1}) - \sigma_{j}^m(t^{m+1})}
\int_{x_j}^{x_{j+1}} u\bigl(t^m, x\bigr)\,\mathrm{d}x
&&& \text{[by~\eqref{eq:16}]}
\end{flalign*}
From now on, we refer to $U$ as the discrete piecewise constant approximation of the real solution $u$ whose cell averages $U^m_j$ are prescribed by the scheme.
Explicitly, we have 
\begin{equation}\label{eq:40}
  U(t,x) = \sum_{j,m} U^m_j \chi_{C_j}(x) \chi_{[t^m, t^{m+1}]}(t), \qquad U^m(x) \coloneqq U(t^m, x).  
\end{equation}
The first step towards the construction of the numerical scheme is the adequate discretization of the \textit{no-flux} curves between timing $[t^m, t^{m+1}]$. To this aim, consider
\begin{equation}\label{eq:50}
    V_j^m \coloneqq v \left( U^m_j
\right) \sim v\left( u(t^m, x_j) \right)
\end{equation}
and let $\bar x^{m+1}_j$ be the approximation of the displacement of $x_j$ along the curve $\sigma_j^m$ at the time $t^{m+1}$:
\begin{equation}\label{eq:39}
    \bar x_j^{m+1} \coloneqq x_j + \Delta t \, V^m_j \sim \sigma_j^{m}(t^{m+1}).
\end{equation}
To establish the well-posedness of the scheme, the time interval $\Delta t$ is chosen upon a CFL-type condition which prevents the interaction between  \textit{no-flux} curves.
Indeed, the choice 
\begin{equation}\label{eq:20}
\modulo{V_j^m} \leq \frac{1}{8} \frac{\Delta x}{\Delta t}
\end{equation}
immediately implies $
x_{j-\frac{1}{2}} < \bar{x}_j^{m+1} < x_{j + \frac{1}{2}}$.
\begin{remark}\label{remark:1}
    Note that the CFL condition above does not rely on the evaluation of the derivative of the flux. Moreover, we can make the choice $\Delta t = \frac{1}{8}\norma{v}^{-1}_{\L\infty([0, u_{max}]; \reali)}\Delta x$.
\end{remark}

\begin{figure}
    \centering
 \begin{tikzpicture}[xscale=1.1, yscale=0.6]
  \node at (-1.5, 0) {$t^m$};
  \node[above] at (11, 0) {$x$};
  \draw[->](0,0)--(11,0);

\draw (1,-0.1) -- (1, 0.1);
\draw (2,-0.1) -- (2, 0.1);   \node[below] at (2, 0) {$x_{j-\frac{3}{2}}$};
\draw (3,-0.1) -- (3, 0.1);   \node[below] at (3, 0) {$x_{j-1}$};
\draw (4,-0.1) -- (4, 0.1);   \node[below] at (4, 0) {$x_{j-\frac{1}{2}}$};
\draw (5,-0.1) -- (5, 0.1);   \node[below] at (5, 0) {$x_j$};
\draw (6,-0.1) -- (6, 0.1);   \node[below] at (6, 0) {$x_{j+\frac{1}{2}}$};
\draw (7,-0.1) -- (7, 0.1);   \node[below] at (7, 0) {$x_{j+1}$};
\draw (8,-0.1) -- (8, 0.1);   \node[below] at (8, 0) {$x_{j+\frac{3}{2}}$};
\draw (9,-0.1) -- (9, 0.1);
\draw (10,-0.1) -- (10, 0.1);

  \node at (-1.5, 4) {$t^{m+1}$};
  \node[above] at (11, 4) {$x$};
  \draw[->](0,4)--(11,4);

\draw (1,3.9) -- (1, 4.1);
\draw (2,3.9) -- (2, 4.1);   \node[below] at (2, 4) {$x_{j-\frac{3}{2}}$};
\draw (3,3.9) -- (3, 4.1);   \node[below] at (3, 4) {$x_{j-1}$};
\draw (4,3.9) -- (4, 4.1);   \node[below] at (4, 4) {$x_{j-\frac{1}{2}}$};
\draw (5,3.9) -- (5, 4.1);   \node[below] at (5, 4) {$x_j$};
\draw (6,3.9) -- (6, 4.1);   \node[below] at (6, 4) {$x_{j+\frac{1}{2}}$};
\draw (7,3.9) -- (7, 4.1);   \node[below] at (7, 4) {$x_{j+1}$};
\draw (8,3.9) -- (8, 4.1);   \node[below] at (8, 4) {$x_{j+\frac{3}{2}}$};
\draw (9,3.9) -- (9, 4.1);
\draw (10,3.9) -- (10, 4.1);

\draw[dashed, magenta] (5,0) -- (5.7, 4);
\draw[snake it, cyan] (5.07,0.3) -- (5.3,4);
\draw[cyan] (5.0,0) -- (5.07,0.3);

\draw[magenta] (5.7,3.9) -- (5.7, 4.1);   \node[above] at (6, 4) {$\color{magenta}\bar x_{j}^{m+1}$};
\draw[cyan] (5.3,3.9) -- (5.3, 4.1);   \node[above] at (4.6, 4) {$\color{cyan}{\sigma_j^m(t^{m+1})}$};

\draw[dashed, magenta] (3,0) -- (2.2, 4);
\draw[snake it, cyan] (2.96,0.3) -- (2.6,4);
\draw[cyan] (3,0) -- (2.96, 0.3);

\draw[magenta] (2.2,3.9) -- (2.2, 4.1);   \node[above] at (2.2, 4) {$\color{magenta}\bar x_{j-1}^{m+1}$};

\draw[dashed, magenta] (7,0) -- (6.7, 4);
\draw[snake it, cyan] (6.95,0.3) -- (7.4,4);
\draw[cyan] (7,0) -- (6.95, 0.3);

\draw[magenta] (6.7,3.9) -- (6.7, 4.1);   \node[above] at (7, 4) {$\color{magenta}\bar x_{j+1}^{m+1}$};

\node[above] at (6, 1.7) {$\color{cyan}D_j^m$};
\node[above] at (4, 1.7) {$\color{cyan}D_{j-1}^m$};

 \end{tikzpicture}
 \caption{Scheme of the moving regions $D_j^m$ "impermeable" to the flux defined in~\eqref{eq:49}. In cyan, we plotted the \emph{no-flux} curves described by~\eqref{eq:13}; in magenta we add the linear approximation of the curves (see~\eqref{eq:50} and~\eqref{eq:39}).
}
\end{figure}
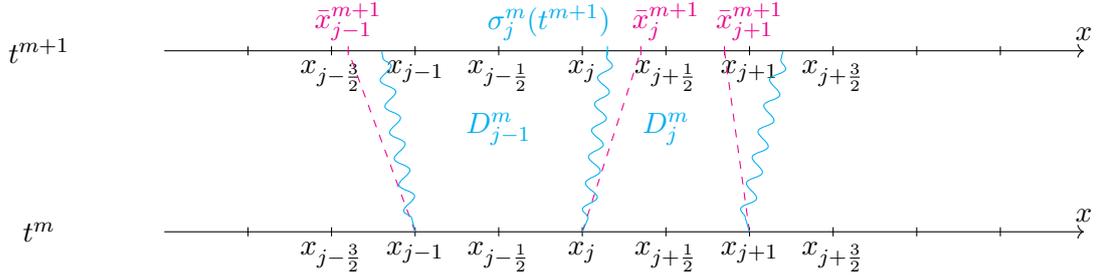
We are now ready to describe the scheme defining the approximate solution $U$ by following its spatial average on the moving region which mimic $D_j^m$. Indeed, we define  
\begin{flalign*}
&& h_{j}^{m+1} \coloneqq \bar x_{j+1}^{m+1} - \bar x_{j}^{m+1} = \Delta x + \Delta t \left( V^m_{j+1} - V^m_j\right) &&& \text{[by~\eqref{eq:39}]} 
\end{flalign*}
and claim that 
\begin{align}
\nonumber
    u_{j+\frac{1}{2}}^{m+1}
    \sim
    &
    \frac{1}{h^{m+1}_j} \int_{ x_j}^{ x_{j+1}} U(t^m, x) \dd x
    \\
    \nonumber
    =
    &
    \frac{1}{h^{m+1}_j}
    \left[
    \int_{ x_j}^{ x_{j+\frac{1}{2}}} U(t^m, x) \dd x + \int_{ x_{j+\frac{1}{2}}}^{ x_{j+1}} U(t^m, x) \dd x
    \right]
    \\
    \label{eq:26}
    =
    &
    \frac{1}{h^{m+1}_j} \frac{\Delta x}{2}
    \left(
    U_j^m  + 
    U_{j+1}^m  
    \right) \eqqcolon U^{m+1}_{j+\frac{1}{2}}.
\end{align}
Consequently, one obtains
\begin{align*}
\frac{1}{\Delta x} \int_{x_{j-\frac{1}{2}}}^{x_{j+\frac{1}{2}}} U(t^{m+1}, x) \dd x
=
&
\frac{1}{\Delta x} 
\left[\int_{x_{j-\frac{1}{2}}}^{\bar x_j^{m+1}} U(t^{m+1}, x) \dd x
+
\int_{\bar x_j^{m+1}}^{x_{j+\frac{1}{2}}} U(t^{m+1}, x) \dd x
\right]
\\
\sim
&
\frac{1}{\Delta x} 
\left[\int_{x_{j-\frac{1}{2}}}^{\bar x_j^{m+1}} U^{m+1}_{j-\frac{1}{2}} \dd x
+
\int_{\bar x_j^{m+1}}^{x_{j+\frac{1}{2}}} U^{m+1}_{j+\frac{1}{2}} \dd x
\right]
\\
=
&
\frac{1}{\Delta x} 
\left[ U^{m+1}_{j-\frac{1}{2}} \left(
\frac{\Delta x}{2} + \Delta t V_j^m
\right)
+
U^{m+1}_{j+\frac{1}{2}} \left(
\frac{\Delta x}{2} - \Delta t V_j^m
\right)
\right].
\end{align*}

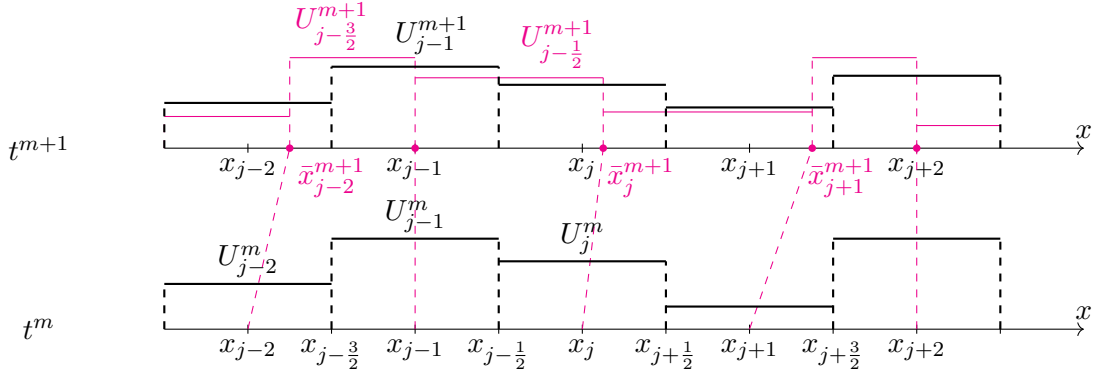
\begin{figure}
    \centering
 \begin{tikzpicture}[xscale=1.1, yscale=0.6]
  \node at (-1.5, 0) {$t^m$};
  \node[above] at (11, 0) {$x$};
  \draw[->](0,0)--(11,0);

\draw (1,-0.1) -- (1, 0.1);   \node[below] at (1,0) {$x_{j-2}$};
\draw (2,-0.1) -- (2, 0.1);   \node[below] at (2, 0) {$x_{j-\frac{3}{2}}$};
\draw (3,-0.1) -- (3, 0.1);   \node[below] at (3, 0) {$x_{j-1}$};
\draw (4,-0.1) -- (4, 0.1);   \node[below] at (4, 0) {$x_{j-\frac{1}{2}}$};
\draw (5,-0.1) -- (5, 0.1);   \node[below] at (5, 0) {$x_j$};
\draw (6,-0.1) -- (6, 0.1);   \node[below] at (6, 0) {$x_{j+\frac{1}{2}}$};
\draw (7,-0.1) -- (7, 0.1);   \node[below] at (7, 0) {$x_{j+1}$};
\draw (8,-0.1) -- (8, 0.1);   \node[below] at (8, 0) {$x_{j+\frac{3}{2}}$};
\draw (9,-0.1) -- (9, 0.1);   \node[below] at (9, 0) {$x_{j+2}$};
\draw (10,-0.1) -- (10, 0.1);

  \node at (-1.5, 4) {$t^{m+1}$};
  \node[above] at (11, 4) {$x$};
  \draw[->](0,4)--(11,4);
\draw (1,3.9) -- (1, 4.1);   \node[below] at (1, 4) {$x_{j-2}$};

\draw (3,3.9) -- (3, 4.1);   \node[below] at (3, 4) {$x_{j-1}$};

\draw (5,3.9) -- (5, 4.1);   \node[below] at (5, 4) {$x_j$};

\draw (7,3.9) -- (7, 4.1);   \node[below] at (7, 4) {$x_{j+1}$};

\draw (9,3.9) -- (9, 4.1);   \node[below] at (9, 4) {$x_{j+2}$};

\draw[dashed, magenta] (1,0) -- (1.5,4);

\node[circle, fill, inner sep=1pt, magenta] at (1.5,4){}; 
\node[below] at (2, 4) {\color{magenta}$\bar x^{m+1}_{j-2}$};

\draw[dashed, magenta] (3,0) -- (3, 4);

\node[circle, fill, inner sep=1pt, magenta] at (3,4){}; 

\draw[dashed, magenta] (5,0) -- (5.25, 4);

\node[circle, fill, inner sep=1pt, magenta] at (5.25,4){}; 
\node[below] at (5.7, 4) {\color{magenta}$\bar x^{m+1}_{j}$};

\draw[dashed, magenta] (7,0) -- (7.75, 4);

\node[circle, fill, inner sep=1pt, magenta] at (7.75,4){}; 
\node[below] at (8.1, 4) {\color{magenta}$\bar x^{m+1}_{j+1}$};

\draw[dashed, magenta] (9,0) -- (9, 4);
\node[circle, fill, inner sep=1pt, magenta] at (9,4){};

\draw[dashed,thick] (0,0) -- (0,1);
\draw[thick] (0,1) -- (2,1);
\draw[dashed,thick] (2,0) -- (2,2);
\draw[thick] (2,2) -- (4,2);
\draw[dashed,thick] (4,0) -- (4,2);
\draw[thick] (4,1.5) -- (6,1.5);
\draw[dashed,thick] (6,0) -- (6,1.5);
\draw[thick] (6,0.5) -- (8,0.5);
\draw[dashed,thick] (8,0) -- (8,2);
\draw[thick] (8,2) -- (10,2);
\draw[dashed,thick] (10,0) -- (10,2);

\node at (1, 1.5) {$U^{m}_{j-2}$};
\node at (3, 2.5) {$U^{m}_{j-1}$};
\node at (5, 2) {$U^{m}_{j}$};

\node at (2.0, 6.7) {\color{magenta}$U^{m+1}_{j-\frac{3}{2}}$};
\node at (4.7, 6.3) {\color{magenta}$U^{m+1}_{j-\frac{1}{2}}$};

\node at (3.2, 6.5) {$U^{m+1}_{j-1}$};

\draw[dashed,magenta] (0,4) -- (0,4.7);
\draw[magenta] (0,4.7) -- (1.5,4.7);
\draw[dashed, magenta] (1.5,4) -- (1.5,6);
\draw[magenta] (1.5,6) -- (3,6);
\draw[dashed, magenta] (3,4) -- (3,6);
\draw[magenta] (3,14/9+4) -- (5.25,14/9+4);
\draw[dashed,magenta] (5.25,4) -- (5.25,14/9+4);
\draw[magenta] (5.25,4.8) -- (7.75,4.8);
\draw[dashed,magenta] (7.75,4) -- (7.75,6);
\draw[magenta] (7.75,6) -- (9,6);
\draw[dashed, magenta] (9,4) -- (9,6);
\draw[magenta] (9,4.5) -- (10,4.5);

\draw[dashed,thick] (0,4) -- (0,5);
\draw[thick] (0,5) -- (2,5);
\draw[dashed,thick] (2,4) -- (2,5.8);
\draw[thick] (2,5.8) -- (4,5.8);
\draw[dashed,thick] (4,4) -- (4,5.8);
\draw[thick] (4,5.4) -- (6,5.4);
\draw[dashed,thick] (6,4) -- (6,5.4);
\draw[thick] (6,4.9) -- (8,4.9);
\draw[dashed,thick] (8,4) -- (8,5.6);
\draw[thick] (8,5.6) -- (10,5.6);
\draw[dashed,thick] (10,4) -- (10,5.6);

 \end{tikzpicture}
 \caption{Representation of the Eulerian-Lagrangian scheme~\eqref{eq:21}. In black we plot the solution $U$ at times $t^m$ and $t^{m+1}$; in magenta dashed we report the approximation of \emph{no-flux} curves and the values $U^{m+1}_{j+\frac{1}{2}}$ defined in~\eqref{eq:26}. 
}
\end{figure}
Hence, we are motivated to set the following numerical scheme
\[
U^{m+1}_j = \frac{1}{\Delta x} 
\left[ U^{m+1}_{j-\frac{1}{2}} \left(
\frac{\Delta x}{2} + \Delta t V_j^m
\right)
+
U^{m+1}_{j+\frac{1}{2}} \left(
\frac{\Delta x}{2} - \Delta t V_j^m
\right)
\right]
\]
which, substituting the values $U_{j+\frac{1}{2}}^{m+1}$ obtained in~\eqref{eq:26} and 
\[
\frac{\Delta x}{2} + \Delta t V_j^m = \frac{1}{2} \left(h^{m+1}_{j-1} + \Delta t (V_j^m+V_{j-1}^m) \right), \quad 
\frac{\Delta x}{2} - \Delta t V_j^m = \frac{1}{2} \left(h^{m+1}_{j} - \Delta t (V_{j+1}^m+V_{j}^m) \right),
\]
explicitly reads 
\begin{equation}\label{eq:21}
U^{m+1}_j
=
\frac{U^m_{j-1}+2U^m_j + U^m_{j+1}}{4} +
\frac{\Delta t}{4} \left[
\mathcal{F}^m_j
- 
\mathcal{F}^m_{j+1}
\right]
\end{equation}
with the notation 
\begin{equation*}
    \mathcal{F}_j^m \coloneqq \frac{1}{h^{m+1}_{j-1}} (U_j^m + U_{j-1}^m)(V_j^m + V_{j-1}^m).
\end{equation*}
The Eulerian-Lagrangian scheme is extended to the case of nonlocal conservation laws in \cite{abreu_nonloc}.
We strengthen a little the hypothesis therein required, assuming the maximal density $u_{max}=1$ and that:
\begin{enumerate}[label=$\bm{(v_\Delta)}$]
\item \label{hyp:v_num} 
$v(r) = (1-r)$ for all $r \in [0,1]$; 
\end{enumerate}

\begin{enumerate}[label=$\bm{(\eta_\Delta)}$]
\item \label{hyp:eta_num}
$\eta \in \C2([-1,0]; \reali)$ non-decreasing function such that $\norma{\eta}_{\L1(\reali; \reali)}=1$ and $\eta(-1)=0$.
\end{enumerate}

As before, we adopt the notation $\eta_H$ to denote the kernel function rescaled of a parameter $H>0$, i.e. $\eta_H(x) \coloneqq \frac{1}{H}\eta \left(\frac{x}{H}\right)$.
The main difference with the scheme illustrated above lies in the appropriate description of the convolution term. To this aim, observe that
$\text{supp}\,\eta_H \subset \left[-N_H \Delta x, 0\right]$ where $ N_H \coloneqq\lceil \frac{H}{\Delta x}\rceil$
and $ \lceil\cdot\rceil$ is the ceiling function and introduce the following discrete convolution $*_\Delta$:
\begin{equation}\label{eq:27}
    f*_\Delta \eta_H \coloneqq \sum_{j}\left(f*_{\Delta} \eta_H \right)_j \chi_{C_j}, \quad 
    \left(f*_{\Delta} \eta_H \right)_j\coloneqq  
    \frac{\Delta x}{\sum_{i=0}^{-N_H +1} \eta_H(x_i) \Delta x}\sum_{k=0}^{-N_H+1} f_{j-k} \, \eta_H (x_k)
\end{equation}
with $f \in \L\infty(\reali; \reali)$ piecewise constant function of the form $f = \sum_{j} f_j \chi_{C_j}$. Observe that in the case $0 \leq H \leq \Delta x$ the discrete convolution reduces to the pointwise evaluation: indeed, in that case $N_H=1$ and for all $j \in \mathbb{Z}$
\[
\left(f*_{\Delta} \eta_H \right)_j = \frac{1}{\sum_{i=0}^{-N_H +1} \eta_H(x_i)} f_{j} \, \eta_H (0) = f_j.
\]
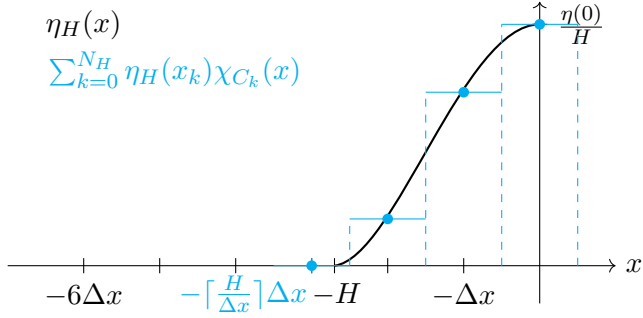
\begin{figure}
\begin{minipage}{0.55\textwidth}
\begin{tikzpicture}[scale=1]
  \draw[->] (-7,0) -- (1,0) node[right] {$x$};
  \draw[->] (0,-0.5) -- (0,3.5);

  \draw (-6,0.1) -- (-6,-0.1) node[below] {$-6\Delta x$};
  \draw (-1,0.1) -- (-1,-0.1) node[below] {$-\Delta x$};
  \draw (-2,0.1) -- (-2,-0.1);
  \draw (-3,0.1) -- (-3,-0.1);
  \draw (-4,0.1) -- (-4,-0.1);
  \draw (-5,0.1) -- (-5,-0.1);

  \draw[black] (-0.1, 3.2) -- (0.1, 3.2 )node[right] {$\frac{\eta(0)}{H}$};

  \draw[below,black] (-2.7,0.1) -- (-2.7,-0.1)node[below] {$\,-H$};

  \draw[cyan] (-3,0.1) -- (-3,-0.1);
  \node[below, cyan] at (-3.9,-0) {$- \lceil \frac{H}{\Delta x}\rceil \Delta x$};

  \draw[thick,black]
  (-2.7,0) .. controls (-1.9,0.1) and (-1 ,3.2) .. (0,3.2);

  \fill[cyan] (-3,0) circle (2pt);
  \fill[cyan] (-2,0.62) circle (2pt);
  \fill[cyan] (-1,2.3) circle (2pt);
  \fill[cyan] (0,3.2) circle (2pt);

  \node at (-6, 3.2) {$\eta_H(x)$};
  \node at (-4.8, 2.6) {\color{cyan}$\sum_{k=0}^{N_H}\eta_H(x_k)\chi_{C_k}(x)$};

  \draw[cyan, dashed] (0.5,0) -- (0.5,3.2);
  \draw[cyan, dashed] (-0.5,0) -- (-0.5,3.2);
  \draw[cyan] (-0.5,3.2) -- (0.5,3.2);

  \draw[cyan, dashed] (-1.5,0) -- (-1.5,2.3);
  \draw[cyan] (-1.5,2.3) -- (-0.5,2.3);

  \draw[cyan, dashed] (-2.5,0) -- (-2.5,0.62);
  \draw[cyan] (-2.5,0.62) -- (-1.5,0.62);

  \draw[cyan] (-3.5,0) -- (-2.5,0);
\end{tikzpicture}
\end{minipage}
\begin{minipage}{0.4\textwidth}
\caption{The kernel function $\eta_H$ and its piecewise constant projection onto the grid.}
\end{minipage}
\end{figure}
The Eulerian Lagrangian scheme for~\eqref{eq:23} reads
\begin{align}\label{eq:28}
\left(U_H\right)^{m+1}_j
= &
\frac{\left(U_H\right)^m_{j-1}+2\left(U_H\right)^m_j + \left(U_H\right)^m_{j+1}}{4} +
\frac{\Delta t}{4} \left[
\left( \mathcal{F}_H\right)^m_j
- 
\left(\mathcal{F}_H\right)^m_{j+1}
\right]
\end{align}
where 
\[
\left(\mathcal{F}_H\right)^m_j = \, 
\frac{1}{\left( h_H \right)^{m+1}_{j-1}} 
\left(\left(U_H\right)_j^m + \left(U_H\right)_{j-1}^m\right)
\left(\left(V_H\right)_j^m + \left(V_H\right)_{j-1}^m\right), 
\]
\[
\left(V_H\right)_j^m = \,
v \left(  \left( \left(U_H\right)^m*_\Delta \eta_{H}
\right)_j\right),
\qquad
\left(h_H\right)^{m+1}_{j} = \, \Delta x +
\Delta t \left( 
\left( V_H\right)^m_{j+1} - \left( V_H\right)^m_{j}
\right).
\]
The convergence of the scheme to the unique solution to~\eqref{eq:23} is proved in \cite[Theorem 2.1]{abreu_nonloc} assuming the same CFL condition~\eqref{eq:20}. Moreover, the numerical method satisfies the maximum principle (\cite[Lemma 3.2]{abreu_nonloc}).
\begin{theorem}[Discrete \emph{nonlocal-to-local} limit]\label{teo:10}
    Assume \ref{hyp:eta_num} and \ref{hyp:v_num} hold and fix $\Delta x >0$. Consider 
    \begin{align}
    \mathcal{U}_{\Delta} \coloneqq \big\{ 
    &
    U_o \in (\L1 \cap \BV)(\reali; [0,1]) \colon \, \exists (U_{o,j})_{j \in \mathbb{Z}} \,\, \text{s.t.}\, \, U_o= \sum_{j \in \mathbb{Z}} U_{o,j}\chi_{C_j}\big\}
\end{align} 
where $C_j = \left[ \left(j-\frac{1}{2}\right)\Delta x , \left(j+\frac{1}{2}\right)\Delta x \right]$. 
For all $U_o \in \mathcal{U}_{\Delta}$ and every sequences $\left\{ U_o^n \right\}_n \subset \mathcal{U}_{\Delta}$ convergent to $U_o$ in $\L1(\reali; \reali)$ and $H_n \to 0^+$, call $U^n $ the discrete solution to~\eqref{eq:23} obtained by the Eulerian-Lagrangian scheme with initial datum $U_o^n$ and convolution kernel $\eta_{H_n}$ defined by the scheme~\eqref{eq:28} and let $U$ defined by the Eulerian-Lagrangian scheme~\eqref{eq:21} applied to the Cauchy problem for the local conservation law~\eqref{eq:24} with initial datum $U_o$.
Then it holds that 
\[
\sup_{t \in [0,T]} \norma{U^n(t,x) - U(t,x)}_{\L1(\reali; \reali)}\xrightarrow[]{n \to + \infty} 0.
\]
\end{theorem}
\begin{proof}
Observe that as $H_n \to 0^+$, we can suppose that $0<H_n \leq \Delta x$. Hence, by definition of discrete convolution $*_\Delta$~\eqref{eq:27}, the Eulerian-Lagrangian scheme for nonlocal conservation laws~\eqref{eq:28} cannot be distinguished from the analogous for local ones~\ref{eq:21}. Hence, the thesis reduces to show that at each time $t \in [0,T]$ the approximate solutions to~\eqref{eq:24} are continuously dependent on the initial datum w.r.t. the $\L1$-norm, which is a property of the scheme, see \cite{abreu_loc}. 
\end{proof}

\subsection{Discrete control problem in the \emph{nonlocal-to-local} limit.}
In the present subsection we finally present the discrete analogue of the control problem discussed before in the continuum case. Again, the main tool to establish a link between minimizers in the \emph{nonlocal-to-local} limit is the $\Gamma$-convergence at the discrete level. 
\begin{theorem}[Discrete $\Gamma$-convergence]\label{teo:8}
    Assume \ref{hyp:eta_num} and \ref{hyp:v_num} hold and fix $M_\Delta, \Delta x > 0$ and $K \subset \subset \reali$. 
    Consider the space
 \vspace{-0.7cm}   
\begin{center}
\begin{equation}\label{eq:34}
\mathcal{U}_{\Delta,ad} \coloneqq
\left\{\, U_o \in (\L1 \cap \BV)(\mathbb{R};[0,1]) \;\middle|\;
\begin{minipage}[t]{0.38\textwidth} 
\raggedright
\vspace{-17pt}
$\exists (U_{o,j})_{j \in \mathbb{Z}}$ s.t. 
$U_o = \sum_{j \in \mathbb{Z}} U_{o,j}\chi_{C_j}$, \\
\vspace{+7pt}
$\tv(U_o) \leq M_\Delta$, \quad $supp(U_o) \subset K$
\end{minipage}
\right\}
\end{equation} 
\end{center}
equipped with the metric induced by the $\L1$ norm. 
Consider the functionals \ref{func_J},\ref{func_K} and \ref{func_I} as in \Cref{teo:4}.
    Consider the functionals $\mathcal{G}_{\Delta, H}: \mathcal{U}_{\Delta, ad} \to \reali$ indexed by $H > 0 $ and $\mathcal{G}_\Delta: \mathcal{U}_{\Delta, ad} \to \reali$ defined as: 
    \begin{align}\label{eq:41}
 \mathcal G_{\Delta, H}(U_o) \coloneqq 
 &
 \mathcal J(\left(U_H\right)(T, \cdot)) + \mathcal K(U_H) +  \mathcal I (U_o) 
 &
\text{with $U_H$ given by~\eqref{eq:28}} ,
 \\
 \label{eq:51}
 \mathcal G_\Delta(U_o) \coloneqq 
 &
 \mathcal J(U(T, \cdot)) + \mathcal K(U) + \mathcal I(U_o) 
 &
\text{with $U$ given by~\eqref{eq:21}} .
\end{align}
Then, the family of functionals $\left\{ \mathcal{G}_{\Delta,H}\right\}_H $ $\Gamma$-converges to $\mathcal{G}_\Delta$ in the limit $H \to  0^+$.
\end{theorem}

\begin{proof}
  The proof is the discrete re-adaptation of the one to \Cref{teo:4} taking advantage of the convergence results in \Cref{teo:10}. We sketch here the proof for completeness. Indeed, the proof is completed once the following conditions are verified: 
  \begin{enumerate}
[label=$\bm{(\Gamma_\Delta1)}$]
\item \label{Gamma conv 1 discrete} 
\textbf{Lower bound inequality.}
For every $U_o \in \mathcal{U}_{\Delta, ad}$ and for every sequence $\{U_{o}^n\}_n \subset \mathcal{U}_{\Delta, ad}$ such that $U_{o}^n \to U_{o}$ in $\L1(\reali; \reali)$, it holds that
\begin{equation*}
    \mathcal G_\Delta(U_{o}) \leq \liminf_{n \to \infty} \mathcal G_{\Delta, H_n}(U^n_{o}).
\end{equation*}
\end{enumerate}
\begin{enumerate}
[label=$\bm{(\Gamma_\Delta2)}$]
\item \label{Gamma conv 2 discrete}
\textbf{Upper bound equality.}
For every $U_{o} \in \mathcal{U}_{\Delta, ad}$ there exists a sequence  $\{U_{o}^n\}_n$ converging to $U_{o}$ in $\mathcal{U}_{\Delta, ad}$ such that
\begin{equation*}
    \mathcal G_\Delta(U_{o}) = \lim_{n \to \infty} \mathcal G_{\Delta, H_n}(U_{o}^n).
\end{equation*}
\end{enumerate}
\textbf{Proof to \ref{Gamma conv 1 discrete}.}
Fix $U_o \in \mathcal{U}_{\Delta, ad}$ and a convergent sequence $\{U^n_o\}_n \subset  \mathcal{U}_{\Delta, ad}$. Then, \Cref{teo:10} yields that $U^n \to U$ in $\L\infty([0,T]; \L1(\reali; \reali))$, hence also in $\L1([0,T]\times \reali; \reali )$. 
Moreover, \Cref{teo:10} yields in particular that $U^n(T) \to U(T)$ in $\L1(\reali; \reali)$.
Then, the continuity of $\mathcal{J},\mathcal{ K}$ and the lower semi-continuity of $\mathcal{I}$ allow to conclude.\\
\textbf{Proof to \ref{Gamma conv 2 discrete}.}
Fix $U_o \in \mathcal{U}_{\Delta, ad}$ and set $U_o^n \coloneqq U_o$ for all $n \in \naturali$. An application of \Cref{teo:10} yields the desired equality, completing the proof.
\end{proof}
After proving the existence of minimizers at the discrete level, by means of the previous $\Gamma$-convergence result, we can prove the claimed convergence (up to subsequence) of minimizers to $\mathcal{G}_{\Delta, H}$ to minimizers to $\mathcal{G}_\Delta$.
\begin{proposition}[Existence of minimizers to the discrete problems]\label{prop:3}
Let $\mathcal J, \mathcal I: \Lloc1(\reali; \reali) \to \reali_{\geq 0}$ and $\mathcal K: \Lloc1([0,T]\times \reali; \reali) \to \reali_{\geq 0}$ be lower semi-continuous functionals w.r.t. the $\Lloc1$ topology. 
For each $H$ fixed, consider the functionals $\mathcal G_{\Delta, H}, \mathcal G_\Delta$ as in~\eqref{eq:41} and~\eqref{eq:51} defined on the space $\mathcal U_{\Delta, ad}$ in~\eqref{eq:34}. Then, there exist minimizers to the control problems
$$ \min_{U_o \in \mathcal{U}_{\Delta, ad}} \mathcal G_{\Delta,H}(U_o), \quad \min_{U_o \in \mathcal{U}_{\Delta,ad}} \mathcal G_\Delta(U_o).
$$
\end{proposition}
\begin{proof}
    The proof directly descends to the property of lower semi-continuity of $\mathcal J, \mathcal K$ and $\mathcal{I}$, the compactness of $\mathcal{U}_{\Delta, ad}$ w.r.t. the $\L1$ topology and the continuous dependence on the initial datum of the solutions generated by the Eulerian-Lagrangian scheme (see \cite[Theorem 5.1]{abreu_nonloc} and \cite{abreu_loc}). 
\end{proof}

\begin{theorem}[Convergence of minimizers to the discrete problems]\label{teo:9}
Given $H_n \to 0^+$, consider the functionals $\mathcal G_{\Delta,H_n}$, $\mathcal G_\Delta$ as in \Cref{teo:8} and let $U_{o}^n$ be a minimizer of $\mathcal{G}_{\Delta,H_n}$ in $\mathcal{U}_{\Delta, ad}$ for all $n \in \naturali$. Then, up to a subsequence, $U_o^n \to U_o$ in $\L1(\reali; \reali)$ where $U_o$ is a minimizer of $\mathcal G_\Delta$ in $\mathcal{U}_{\Delta,ad}$. 
\end{theorem}

\begin{proof}
    The proof directly descends from the compactness of $\mathcal{ U}_{\Delta, ad}$ in $\L1(\reali; \reali)$ and \Cref{teo:2}. 
\end{proof}

\section{Numerical simulations}
In the present section we show an optimization problem for a specific tracking-type functional with distributed observation. 
At first we present a grid convergence analysis for the local and nonlocal control problems which will be investigated by means of the built-in Matlab function \texttt{fmincon}. 
Then we will numerically validate the convergence of minimizers to $\Gamma$-convergent functionals proved in \Cref{teo:9}. At last, we numerically show such convergence in the double limit $(\Delta x, H) \to 0^+$, being $\Delta x $ the spatial mesh and $H$ the parameter of the kernel function as usual. 

Throughout all the simulations we will suppose that the controls belong to admissible set $\mathcal{U}_{\Delta, ad}$ as in~\eqref{eq:34} with the choice of $K \coloneqq [-1,1]$ and  $\eta(x) = 2(x+1) \chi_{[-1,0]}(x)$.  Moreover, we impose absorbing conditions at the boundaries: as $\eta$ is anisotropic and the speed law $v$ is positive, we add one ghost cell on the left of the boundary and as many cells required for the well-posedness of the discrete convolution~\eqref{eq:27} at the right side.
Namely, for each $H, \Delta x >0$ we add $N_H\coloneqq \lceil \frac{H}{\Delta x}\rceil+1$ cells on the right. 
Then, the solution is constantly extended on such ghost cells, attaining the same value of the boundary.

With reference to the function \texttt{fmincon}, for fixed $\Delta x$ we will set the following options
\begin{lstlisting}
options = optimoptions('fmincon', 'MaxFunctionEvaluations', 1e5, 
           'MaxIterations', 1e3,'OptimalityTolerance', dx^2,
           'StepTolerance', dx^3, 'Display','iter')
\end{lstlisting} 
and initiate the algorithm with the projection in $\mathcal{U}_{\Delta, ad}$ of $
u_{init}(x) \coloneqq 0.25\chi_{[0,+\infty)}(x)+0.2.$
\subsection{The local optimization problem}
Consider the tracking-type functional with distributed observation which evaluates the distance in $\L1([0,T]\times [-1,1]; \reali)$ w.r.t. $u^d$, the entropy solution to 
\begin{equation}\label{eq:47}
    \begin{cases}
    \partial_t u^d + \partial_x \left(u^d v(u^d) \right) =0, \\
    u^d(0,\cdot) = (-x^2 +0.25)\chi_{[-0.5, 0.5]}(\modulo{x}) +0.2;
\end{cases}
\end{equation}
found by scheme~\eqref{eq:21} with mesh size $\Delta x^d \coloneqq 0.002$ and initial datum $ U_o^d = \sum_j \left(U_o^d\right)_j \chi_{C_j^d}$ with $\left(U_o^d\right)_j \coloneqq \frac{1}{\Delta x^d} \int_{C_j^d} u^d_o(x) \dd x$, $C_j^d \coloneqq \left[\left(j-\frac{1}{2}\right)\Delta x^d,\left(j+\frac{1}{2}\right)\Delta x^d \right]$. We call $U^d$ the approximation to $u^d$ found through the Eulerian-Lagrangian scheme and we will refer to it as \emph{reference solution}. 
For fixed $\Delta x>0$, we are interested in the minimization over $U_o \in \mathcal U_{\Delta, ad}$ of the functional
\begin{equation}\label{eq:46}
    \mathcal{G}_\Delta(U_o) \coloneqq \mathcal K (U) = \sum_{m=0}^M \sum_{\substack{j \in \mathbb{Z},\\ -1 \leq j\Delta x \leq 1}} \modulo{U^m_j - U^d(m\Delta t, j\Delta x)} \Delta x \Delta t,
\end{equation} 
where $U$ is set by the scheme~\eqref{eq:21}, $\Delta t =\Delta x/2$ and $M\Delta t = T = 0.25$. 
We note that here $\Delta t$ can be set greater than the CFL condition~\eqref{eq:20} as the approximate solutions provided by the scheme remain stable.

We devote ourselves to the study of the optimization problem $\min_{U_o \in \mathcal{U}_{\Delta, ad}} \mathcal{G}_\Delta(U_o)$, presenting a grid convergence analysis:
for different choices of $\Delta x$, we find $U_o^{min} \in \mathcal{U}_{\Delta, ad}$ by the algorithm \texttt{fmincon} and collect the result in Table \ref{table:1}. In particular we report the relative error
\begin{displaymath}
\norma{U_o^{min}- U_o^d }_{\L1([-1, 1]; \reali)} / \norma{ U_o^d }_{\L1([-1, 1]; \reali)},
\end{displaymath}
the evaluation of the functional in correspondence of the found minimizer $\mathcal{G}_{\Delta} \left( U_o^{min} \right)$, the number of iterations of the algorithm and the measure of First Order Optimality.
\begin{table}[h!]
\centering
\begin{tabular}{c c c c c} \hline $\Delta x$ & $\L1$ relative error & Functional value & Iterations & First Order Optimality \\ \hline 0.08 & 3.78e-02 & 8.15e-03 & 17 & 6.35e-03 \\ 0.04 & 9.91e-02 & 1.58e-02 & 11 & 1.58e-03 \\ 0.02 & 3.36e-02 & 4.33e-03 & 39 & 3.75e-04 \\ 0.01 & 5.31e-03 & 3.50e-04 & 191 & 1.50e-02 \\ \hline \end{tabular}
\caption{Grid convergence analysis for the optimization problem $\min_{U_o \in\mathcal U_{ad}} \mathcal{G}_\Delta(U_o)$ with \texttt{'StepTolerance'} equal to $\Delta x^3$.}
\label{table:1}
\end{table}

In \Cref{fig:local_minimizers}, we compare the exact minimizer $U^d_o$ with $U_o^{min}$ found by the algorithm for the choice $\Delta x =0.01$. Moreover, we show the evolution of $U^d$ and $U^{min}$, the corresponding solutions. 
\begin{figure}
    \centering
    \includegraphics[width=0.45\textwidth]{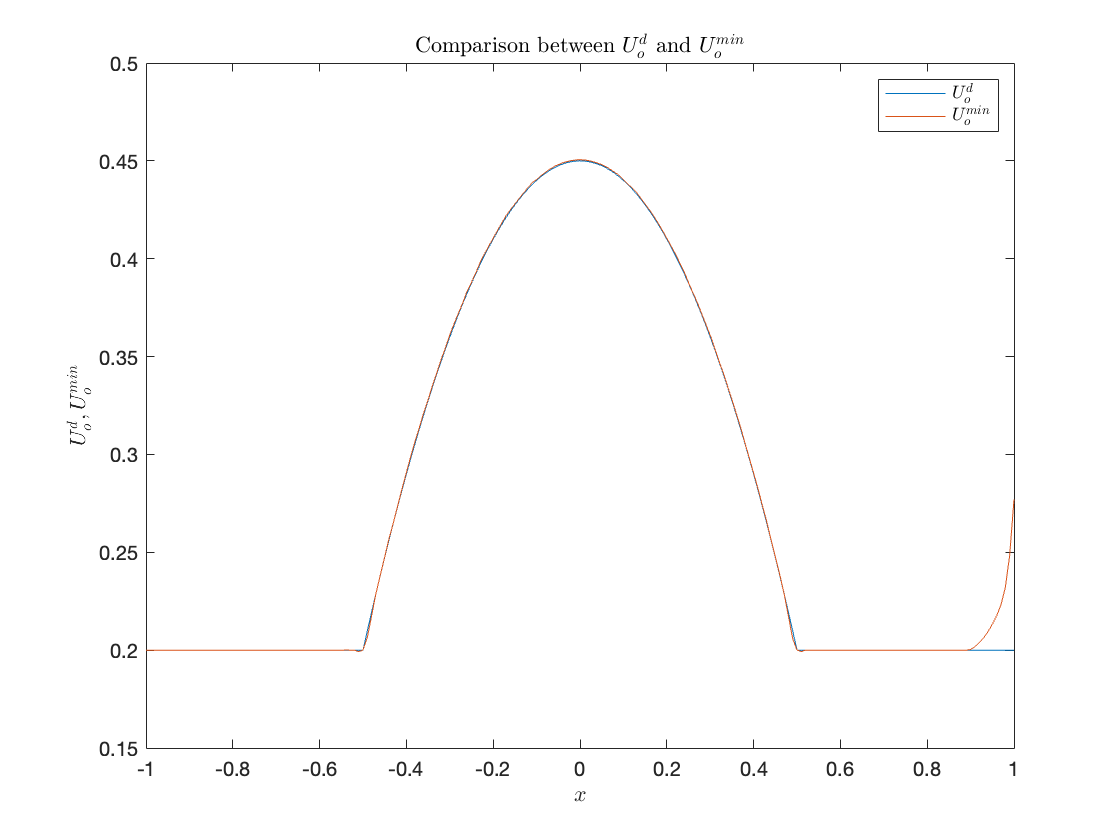}
    \includegraphics[width=0.45\textwidth]{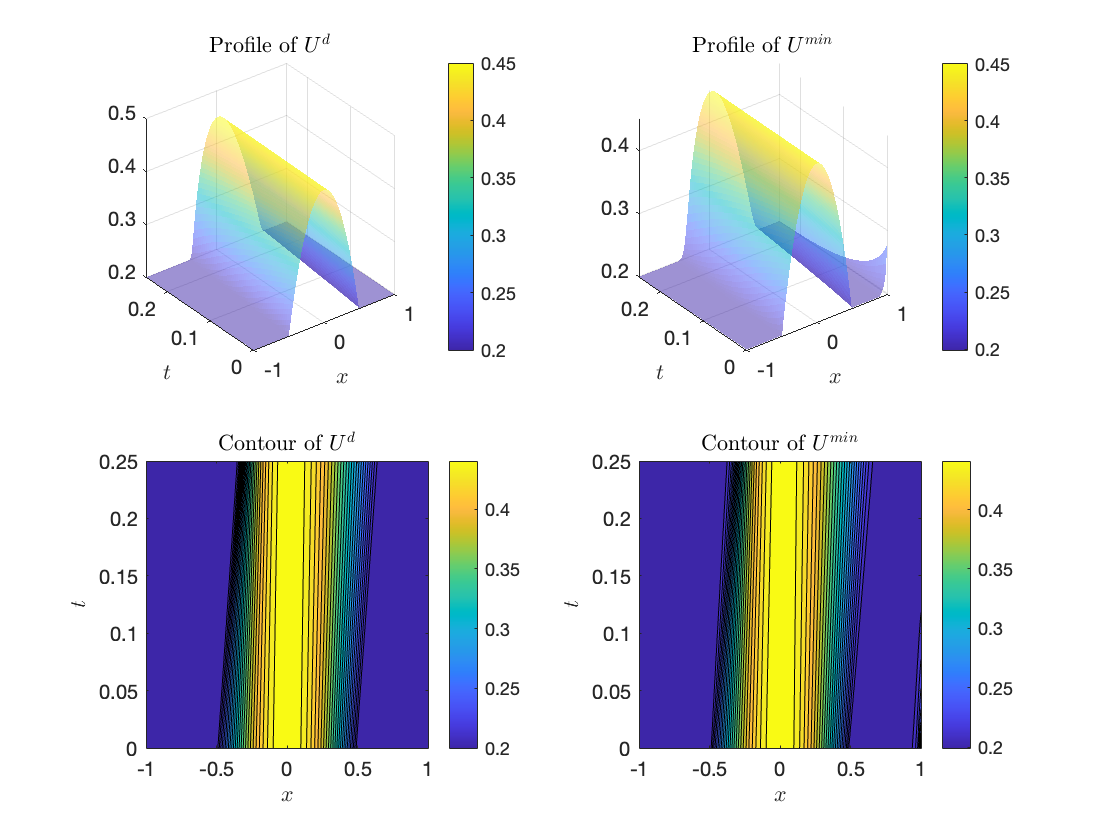}
    \caption{On the left,  we present the comparison between the real minimizer to $\mathcal{G}_\Delta$ (see~\eqref{eq:46}) $U_{o}^d$ and approximation $U_{o}^{min}$ in the case $\Delta x= 0.01$ and \texttt{'StepTolerance'} equal to $\Delta x^3$. The discrepancy at the right boundary may be justified by the fact that there the initial datum is rapidly leaving the domain, thus has no contribution in the objective functional, while on the same time the algorithm is initialized by the value 0.45 on that boundary.
    On the right, we show the evolution in time of $U^d$ and $U^{min}$, the last being the solution to~\eqref{eq:24} correspondent to the approximate optimal initial datum $U_o^{min}$. 
    }
    \label{fig:local_minimizers}
\end{figure}

Related to the options chosen for \texttt{fmincon}, due to the non-differentiability of the objective functional $\mathcal{G}_\Delta$, we are motivated to set \texttt{'StepTolerance'} equal to $\Delta x^3$. In fact, the choice \texttt{'StepTolerance'} equal to $\Delta x^2$ leads to not satisfying results (see\Cref{table:3} and \Cref{fig:local_minimizers_wrong}).
\begin{table}[h!]
\centering
\begin{tabular}{c c c c c} \hline $\Delta x$ & $\L1$ relative error & Functional value & Iterations & First Order Optimality \\ \hline 0.08 & 4.38e-01 & 7.71e-02 & 1 & 4.56e-02 \\ 0.04 & 9.91e-02 & 1.58e-02 & 11 & 1.58e-03 \\ 0.02 & 3.36e-02 & 4.33e-03 & 36 & 4.86e-04 \\ 0.01 & 8.37e-02 & 1.09e-02 & 51 & 2.00e-04 \\ \hline \end{tabular}
\caption{Grid convergence analysis for the optimization problem $\min_{U_o \in\mathcal U_{ad}} \mathcal{G}_\Delta(U_o)$ with \texttt{'StepTolerance'} equal to $\Delta x^2$.}
\label{table:3}
\end{table}
\begin{figure}
    \centering
    \includegraphics[width=0.45\textwidth]{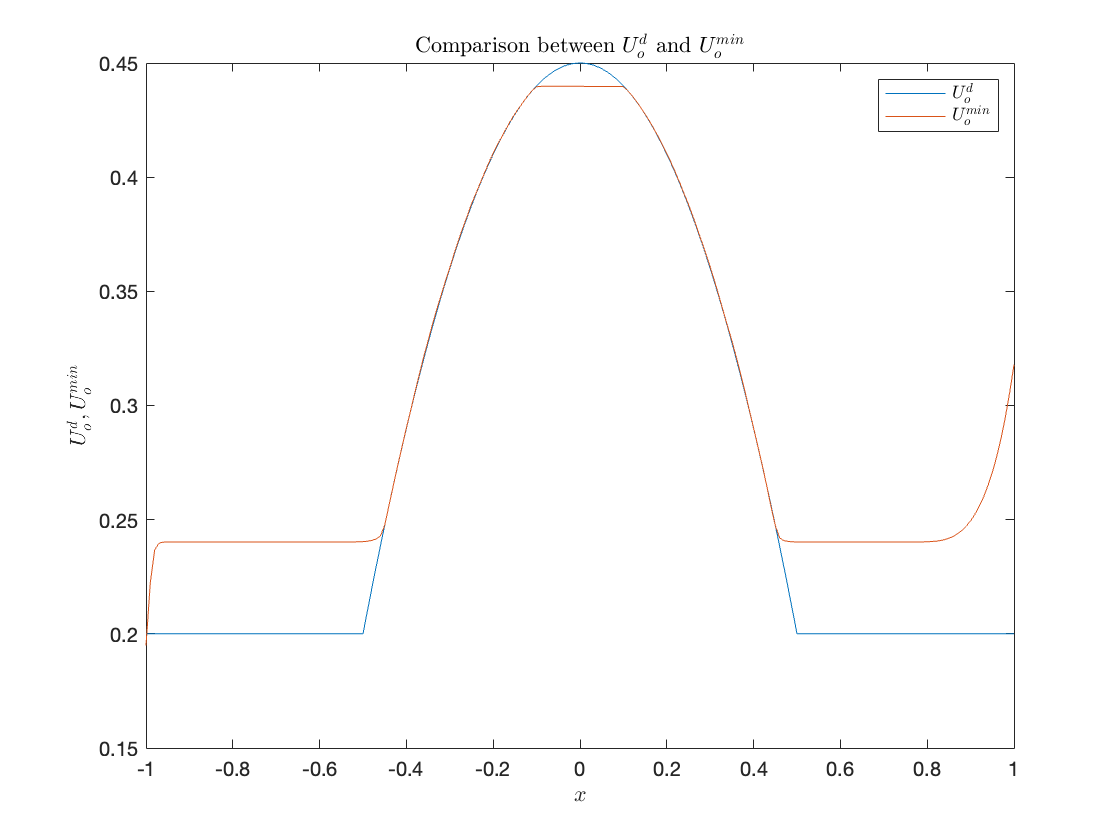}
    \includegraphics[width=0.45\textwidth]{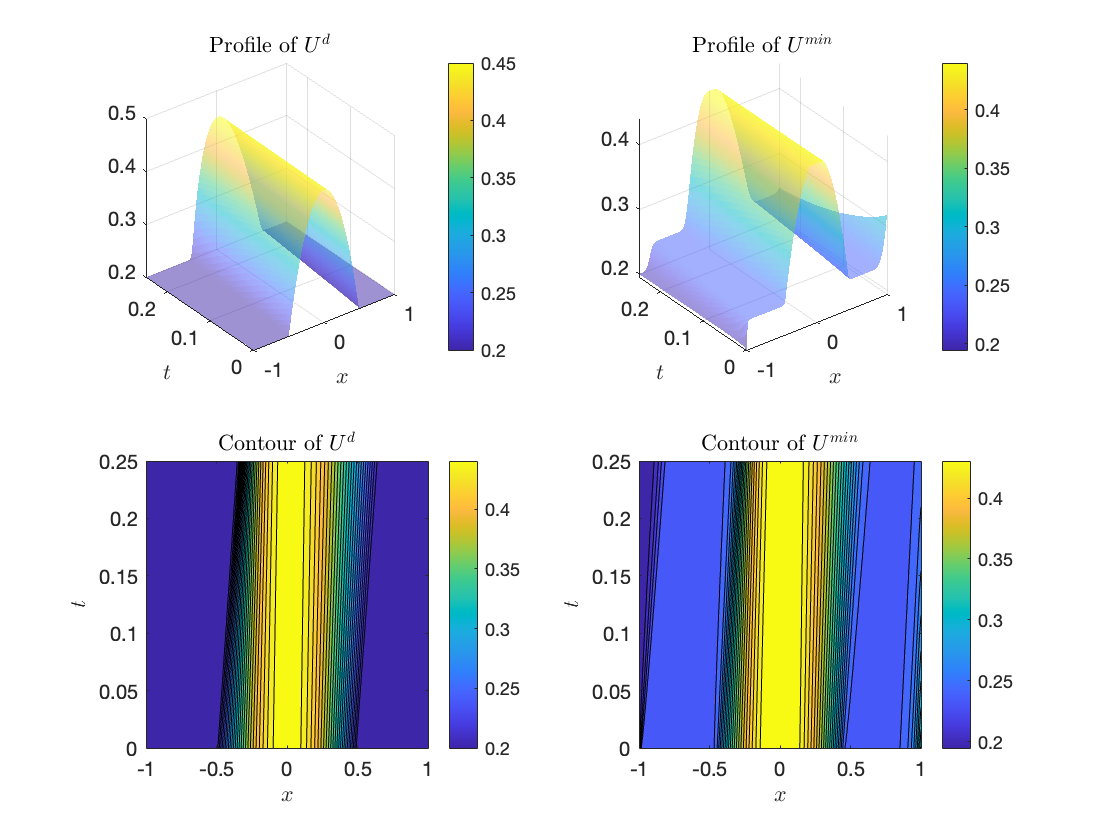}
    \caption{On the left,  we present the comparison between the real minimizer to $\mathcal{G}_\Delta$ (see~\eqref{eq:46}) $U_{o}^d$ and approximation $U_{o}^{min}$ in the case $\Delta x= 0.01$ and \texttt{'StepTolerance'} equal to $\Delta x^2$. 
    On the right, we show the evolution in time of $U^d$ and $U^{min}$, the last being the solution to~\eqref{eq:24} correspondent to the approximate optimal initial datum $U_o^{min}$.
    The minimizer found by Matlab $U_o^{min}$ is very much different from $U^d_o$, maybe due to the non-differentiability of $\mathcal G_\Delta$ and the choice of a piecewise-constant initial datum $u_{init}$. 
    }
\label{fig:local_minimizers_wrong}
\end{figure}
\subsection{The nonlocal optimization problem}
We present here an analogous analysis to the one showed above, turning to a control problem related to the resolution of the nonlocal conservation law~\eqref{eq:23} for $H=0.5$. 
In order to develop a grid convergence analysis we turn to the functional 
\begin{equation}\label{eq:45}
     \widetilde{ \mathcal{G}}_{\Delta, H} (U_o) \coloneqq \widetilde{\mathcal{K}}(U_H) =  \sum_{m=0}^M \sum_{\substack{j \in \mathbb{Z},\\ -1 \leq j\Delta x \leq 1}} \modulo{\left(U_H\right)^m_j - U^d_H(m\Delta t, j\Delta x)} \Delta x \Delta t,
\end{equation}
where $U_H^d$ (referred as \emph{reference solution}) is the discrete approximation (with mesh size $\Delta x^d = 0.002$) of the solution to the Cauchy problem for the nonlocal conservation law
\begin{equation}\label{eq:48}
    \begin{cases}
    \partial_t u_H^d + \partial_x \left(u_H^d v(u_H^d*\eta_H) \right) =0, \\
    u^d_H(0,\cdot) = (-x^2 +0.25)\chi_{[-0.5, 0.5]}(\modulo{x}) +0.2.
\end{cases}
\end{equation}
The discrete initial datum will be denoted $U_{o, H}^d$ as usual. 

In Table \ref{table:2} we collect the results,  reporting the relative error in $\L1([-1,1])$ w.r.t. the exact minimizer $U_{o,H}^d$, the functional value evaluated in the minimizer found by the algorithm, the number of iterations and the measure of first order optimality. 
\begin{table}[ht!]
\centering
\begin{tabular}{c c c c c} \hline $\Delta x$ & $\mathbf{L}^1$ relative error & Functional value & Iterations & First Order Optimality \\ \hline 0.08 & 3.90e-02 & 7.72e-03 & 17 & 3.80e-03 \\ 0.04 & 1.04e-01 & 1.57e-02 & 17 & 1.48e-03 \\ 0.02 & 3.41e-02 & 4.39e-03 & 46 & 3.28e-04 \\ 0.01 & 3.73e-03 & 3.24e-04 & 163 & 8.67e-05 \\ \hline \end{tabular}
\caption{Grid convergence analysis for the optimization problem $\min_{U_o \in\mathcal{U}_{\Delta, ad}} \tilde{\mathcal G}_{\Delta,H}(U_o)$.}
\label{table:2}
\end{table}
We call $U_{o,H}^{min}$ the minimizer found by \texttt{fmincon} for the choice of mesh size $\Delta x=0.01$ and in \Cref{fig:nonlocal_minimizers} we plot the evolution of $U^d_H$ and $U^{min}_H$, the solution found by the scheme~\eqref{eq:28} starting from $U_{o,H}^{min}$.
\begin{figure}
    \centering
\includegraphics[width=0.45\textwidth]{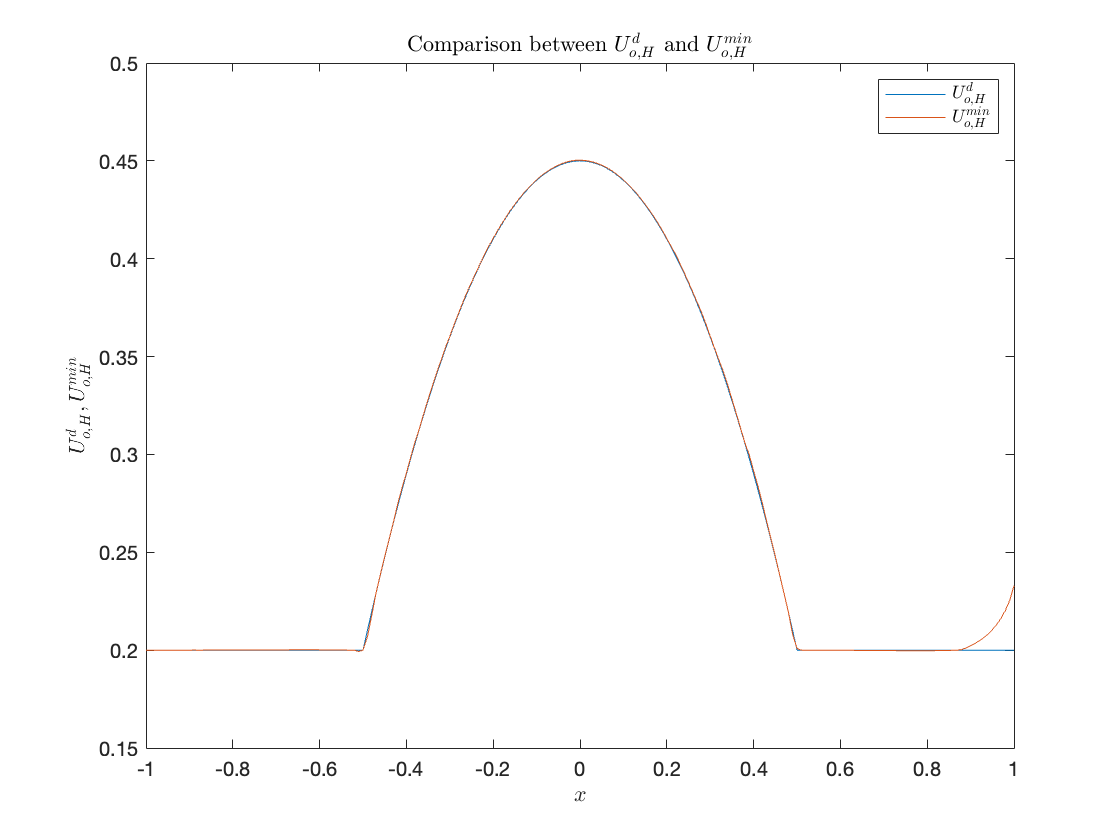}
    \includegraphics[width=0.45\textwidth]{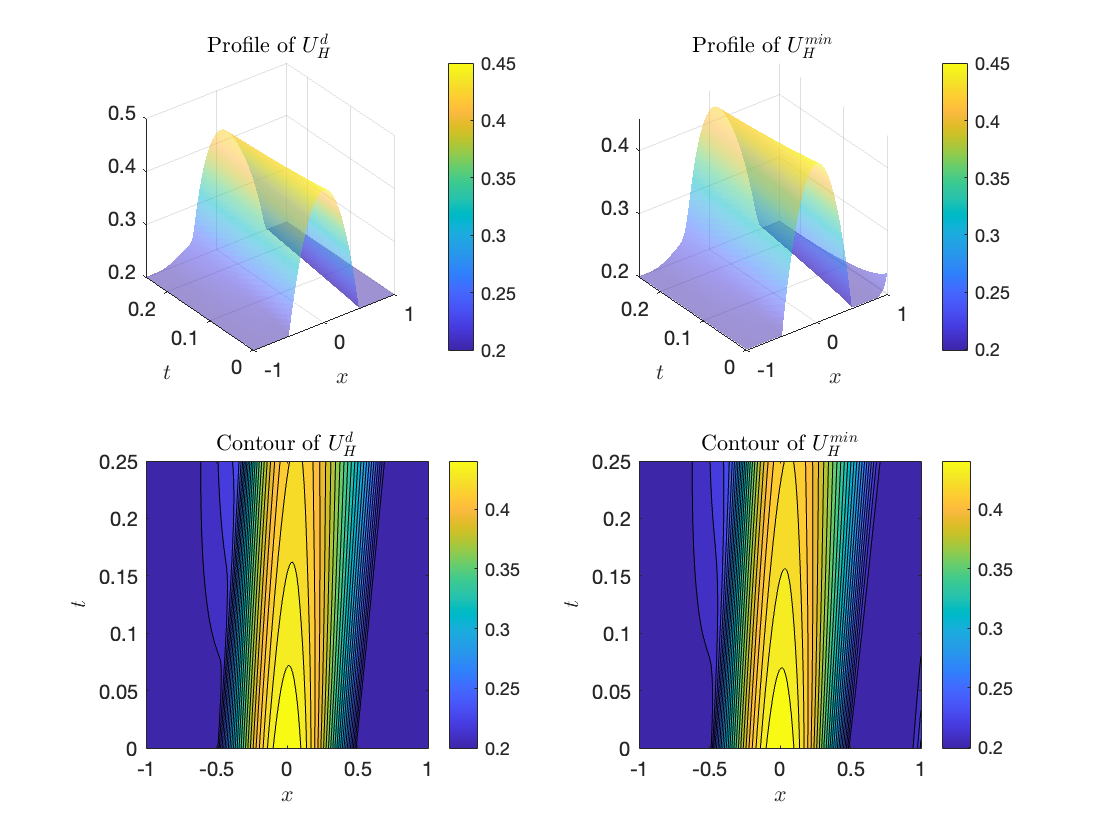}
    \caption{On the left,  we present the comparison between the real minimizer to $\tilde{\mathcal{G}}_{\Delta,H}$ (see~\eqref{eq:45}) $U_{o,H}^d$ and the approximation $U_{o,H}^{min}$ in the case $\Delta x= 0.01$.
    On the right, we show the evolution in time of $U_H^d$ and $U_H^{min}$, the last being the solution to~\eqref{eq:23} correspondent to the approximate optimal initial datum $U_{o,H}^{min}$. }
\label{fig:nonlocal_minimizers}
\end{figure}

\subsection{The discrete $\Gamma$-convergence: convergence of minimizers}
\label{sec:discrete gamma conv}
By the previous analysis we are motivated to choose $\Delta x=0.01$ to guarantee valid results for the optimization problems solved by \texttt{fmincon}. 
So, for this choice of mesh size, we can devote ourselves to the numerical validation of \Cref{teo:9}. 
Consider $\mathcal{G}_\Delta$ as defined in~\eqref{eq:46} and introduce for fixed $H>0$ the functional $\mathcal{G}_{\Delta,H} = \mathcal{K}(U_H)$ where $\mathcal{K}$ is given in~\eqref{eq:46} and $U_H$ is found by the scheme~\eqref{eq:28}.

Let $U_{o,H}^{min}$ be the minimizers to $\mathcal{G}_{\Delta, H}$ found by \texttt{fmincon} and $U_o^{min}$ the minimizer to $\mathcal{G}_\Delta$.

In \Cref{fig:discrete_Gamma_conv} we collect the results, proving the convergence in $\L1([-1,1])$. 
In the left figure we plot in red $U^d_o$ (the exact minimizer to $\mathcal{G}_\Delta$), in black $U^{min}_{o}$ (the minimizer to $\mathcal{G}_\Delta$ found by \texttt{fmincon}) and $U^{min}_{o, H}$ (the minimizers to $\mathcal{G}_{\Delta,H}$ found by \texttt{fmincon}). In the middle figure we zoomed the previous picture at $x=0.5$. At last, in the right figure, we report the relative error $\norma{U^{min}_{o,H} - U^{min}_o}_{\L1([-1,1];\reali )} / \norma{U_o^{min}}_{\L1([-1,1]; \reali)}$. 
Note that for the choice $H=0.005 < \Delta x= 0.01$, the correspondent relative error is zero, due to the fact that under the threshold $H \leq \Delta x$, the Eulerian-Lagrangian scheme for the nonlocal conservation law is identically equal to the correspondent one for the local case.

\begin{figure}[ht!]
\includegraphics[width=1\textwidth]{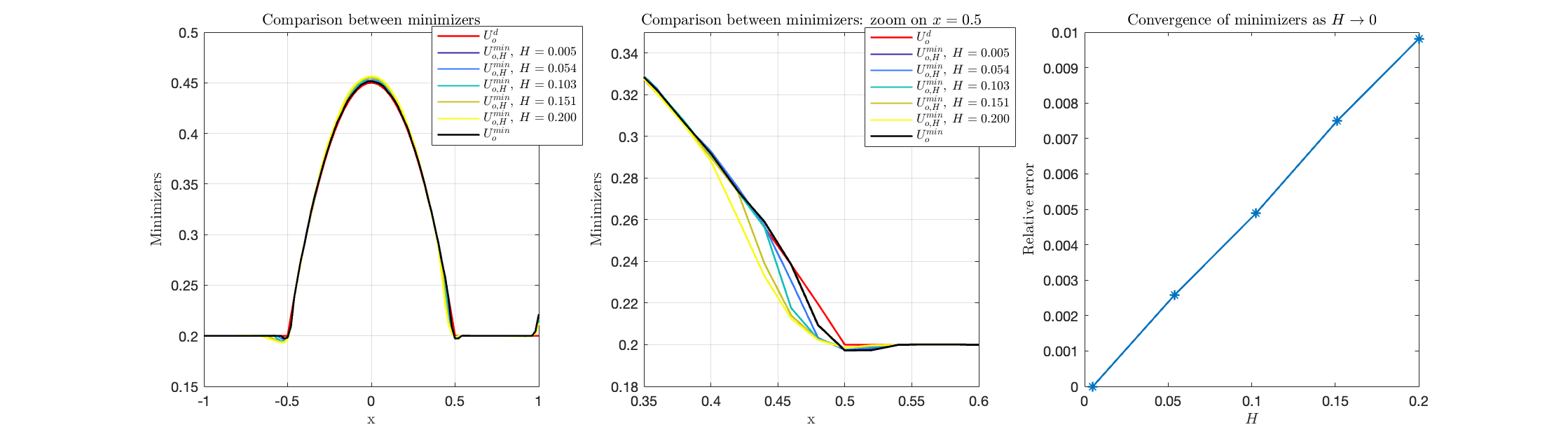}
\caption{ Convergence of minimizers $U^{min}_{o,H}$ to $\mathcal{G}_{\Delta, H}$ to the minimizer $U^{min}_o$ to $\mathcal{G}_\Delta$ as $H \to 0^+$. 
}
\label{fig:discrete_Gamma_conv}
\end{figure}

\subsection{Convergence of minimizers in the limit $\Delta x, H \to 0$}
\label{sec:diagonal convergence}
The previous section is devoted to illustrate numerically the convergence of minimizers to $\Gamma$-convergent functionals in the discrete regime (\Cref{teo:9}). On the other hand, in \Cref{teo:6} we also proved the analogous result at the continuum level. 
Then, the natural question arises whether this limit holds in the simultaneous limit $\Delta x,H \to 0$ (see dashed line in \Cref{scheme:diagonal}). 
We present here a numerical evidence which support the validity of this \emph{double} limit for two possible choices of sequences, $\Delta x= H/2$ and $\Delta x = H^{1.1}$ (upper and lower row respectively in \Cref{fig:diagonal Gamma conv}). 
For each choice $\Delta x, H$ we consider $U^{min}_{o,H}$ the minimizer to $\mathcal{G}_{\Delta, H}$ as in \Cref{sec:discrete gamma conv}. In the left figures we plot $U^{min}_{o,H}$ and $U^d_o$; in the central figures we report the relative error $\norma{U^{min}_{o,H} - U_o^d}_{\L1([-1,1]; \reali)}/ \norma{ U_o^d}_{\L1([-1,1]; \reali)}$; on the right we show $\mathcal{G}_{\Delta, H} (U^{min}_{o,H})$. 
This numerical analysis, which seems to support the convergence claimed above, should be completed by a rigorous proof similar to the one in \cite{zbMATH07852679}.
\begin{figure}[ht!]
\centering
\includegraphics[width=\textwidth]{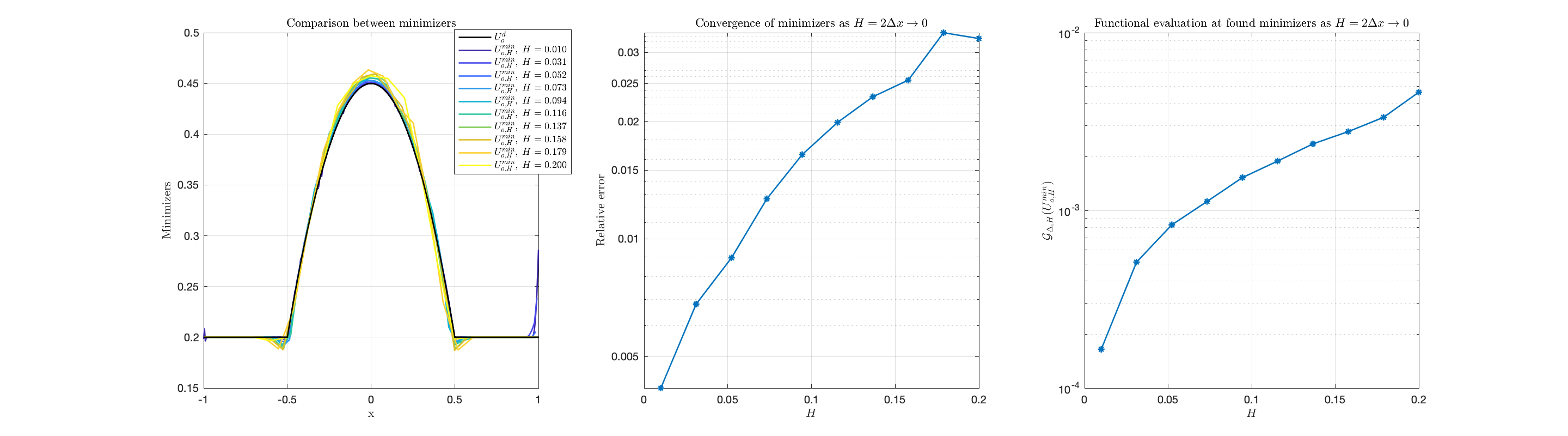}
\includegraphics[width=\textwidth]{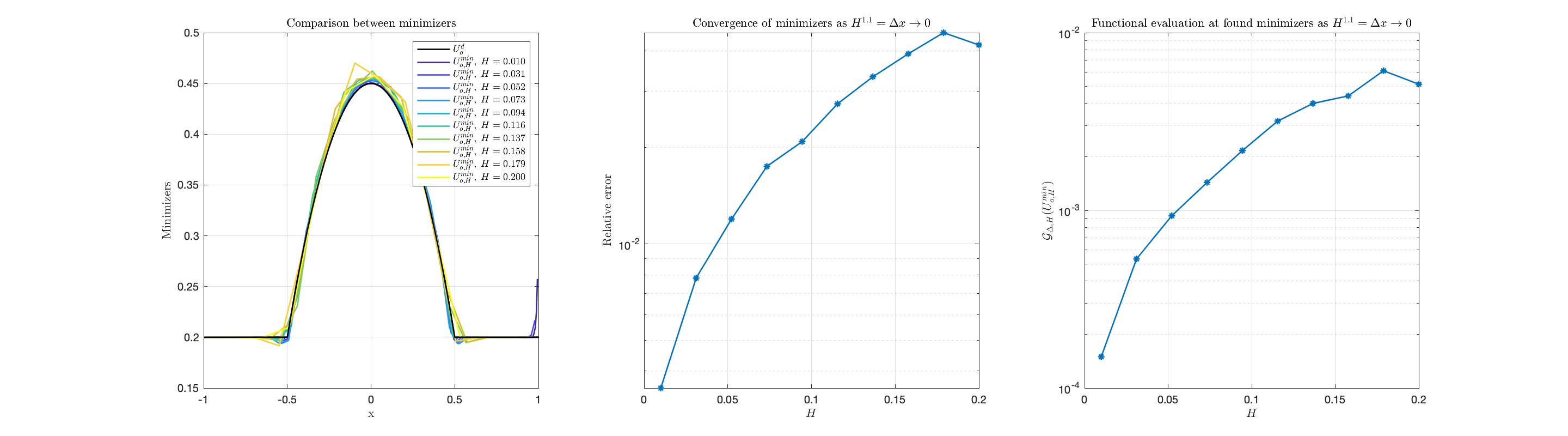}
\caption{Convergence of minimizers to $\mathcal{G}_{\Delta, H}$ to $U^d_o$ in the limit $\Delta x, H \to 0$ for the choices $\Delta x= H/2$ (upper row) and $\Delta x = H^{1.1}$ (lower row), with $H=\texttt{linspace}(0.01,0.1,10)$.
The algorithm \texttt{fmincon} is run with \texttt{'StepTolerance'} $=(\min\{\Delta x\})^3$  and \texttt{'OptimalityTolerance'} = $(\min\{\Delta x\})^2$, where $\min\{\Delta x\}$ denotes the smallest choice of mesh size $\Delta x$, equal to 5e-03 in the first case and 6.31e-03 in the latter.}
\label{fig:diagonal Gamma conv}
\end{figure}
\section{Conclusion}
In the present work we showed that, under suitable hypotheses, control problems dependent on solutions to (local) conservation laws may be dealt as limit of problems dependent on solution to nonlocal conservation laws in the \emph{nonlocal-to-local} limit. 
In particular, by a $\Gamma$-convergence argument,  in \Cref{teo:6} we can characterize the limit up to subsequence of minimizers to nonlocal control problems as minimizers to local counterpart. 
We then show in \Cref{teo:9} the analogue property at the discrete level exploiting the Eulerian-Lagrangian scheme proposed by \cite{abreu_loc,abreu_nonloc}.

The analysis we prove opens up the path to a large variety of related questions and this work represents a first step in the direction of solving \emph{local} control problem as limit of \emph{nonlocal} ones. 
In particular, this strategy can provide a characterization of minimizers to functionals dependent on solutions to local conservation laws as limit of minimizers to nonlocal analogues, bypassing the non-differentiability of the semigroup of solutions to local conservation laws w.r.t. the initial datum and taking advantage of optimality conditions which could be available in the nonlocal framework (see  \cite{zbMATH05908226}, where however the authors impose regularity assumptions on the kernel function $\eta$ which are incompatible to the analysis here presented).

\section*{Funding}
J.~F. is supported by the German Research Foundation (DFG) through SPP 2410 `Hyperbolic Balance Laws in Fluid Mechanics: Complexity, Scales, Randomness' under grant FR 4850/1-1.
C.~N. was supported by the PNRR project \emph{Ricerca DM 118/2023}, CUP~F13C23000360007 and partly supported by the GNAMPA~2025 project
\emph{Modellli di Traffico, di Biologia e di Dinamica dei Gas Basati
  su Sistemi di Equazioni Iperboliche}.

{\small
\bibliographystyle{plain}
  \bibliography{bibtex_arxiv}
}

\end{document}